\documentclass[11pt]{article}
\usepackage{amssymb}
\usepackage{amsmath}
\usepackage{graphicx}
\usepackage[left=3cm,top=3cm,right=3cm,bottom=5cm]{geometry}

\newtheorem{theorem}{Theorem}

\newtheorem{corollary}[theorem]{Corollary}

\newtheorem{definition}[theorem]{Definition}
\newtheorem{example}[theorem]{Example}

\newtheorem{lemma}[theorem]{Lemma}

\newtheorem{proposition}[theorem]{Proposition}
\newtheorem{remark}[theorem]{Remark}

\newenvironment{proof}[1][Proof]{\noindent\textbf{#1.} }{\ \rule{0.5em}{0.5em}}

\begin{document}
	
\title{Chain recurrence and Selgrade's theorem for affine flows}
\date{}
\author{Fritz Colonius$^{1^*}$ $\,\ $  and  $\,\ $ Alexandre J. Santana$^{2{\dag}}$
\vspace{0,3cm}\\ 
	$^{1}$Institut f\"{u} Mathematik, Universit\"{a}t Augsburg, Universit\"{a}tsstraße 2,\\
	Augsburg, 86159, Germany.\\
	$^{2}$Department of Mathematics, State University of Maringa, \\
	Av. Colombo, 5790,	 Maringa, 87020-900, Brazil.
\vspace{0,5cm}\\
	$^{*}$Corresponding author. E-mail: fritz.colonius@uni-a.de;\\
	$^{\dag}$Contributing author: ajsantana@uem.br;\\
	This author contributed equally to this work} 
\maketitle	
	
\begin{abstract}	
Affine flows on vector bundles with chain transitive base flow are lifted to
		linear flows and the decomposition into exponentially separated subbundles
		provided by Selgrade's theorem is determined. The results are illustrated by
		an application to affine control systems with bounded control range.
\end{abstract}	
	

	
	\maketitle
	
	\section{Introduction\label{Section1}}

	For linear (skew product) flows on vector bundles, Selgrade's theorem
	describes the decomposition into subbundles obtained from the chain recurrent
	components of the induced flow on the projective bundle. This coincides with
	the finest decomposition into exponentially separated subbundles. It is a
	simple observation that affine flows can be lifted to linear flows on an
	augmented state space and the main purpose of the present paper is to connect
	the resulting Selgrade decomposition to properties of the original affine flow.
	
	The theory of linear flows was developed in the second half of the last
	century. We refer, in particular, to Sacker and Sell \cite{SacS}, Salamon and
	Zehnder \cite{SalZ88}, Bronstein and Kopanskii \cite{BroK94}, Johnson, Palmer
	and Sell \cite{JoPS87}; cf. also Kloeden and Rasmussen \cite{KloedeR} and
	Colonius and Kliemann \cite{ColK00, ColK14}. An affine flow on a vector bundle
	$\pi:\mathcal{V}\rightarrow B$ over a compact metric space $B$ is a continuous
	flow $\Psi$ on $\mathcal{V}$ preserving fibers such that the induced maps on
	the fibers are affine. We will only consider (topologically) trivial vector
	bundles of the form $\mathcal{V}=B\times H$, where $H$ is a Hilbert space and
	suppose that the base flow on $B$ is chain transitive.
	
	Selgrade's theorem for linear flows $\Phi$ (Selgrade \cite{Selg75},
	\cite[Theorem 9.2.5]{ColK14}) states that the induced flow $\mathbb{P}\Phi$ on
	the projective bundle $\mathbb{P}\mathcal{V}$ has finitely many chain
	recurrent components (this coincides with the finest Morse decomposition). The
	chain recurrent components define invariant subbundles which yield the finest
	decomposition of $\mathcal{V}$ into exponentially separated subbundles.
	Generalizations include Patr\~{a}o and San Martin \cite{PatSM07} for semiflows
	on fiber bundles, Alves and San Martin \cite{AlvSM16} for principal bundles,
	and Blumenthal and Latushkin \cite{BluL19} for linear semiflows on separable
	Banach bundles.
	
	The essence of our set-up is to lift an affine flow $\Psi$ to a linear flow
	$\Psi^{1}$. When we apply Selgrade's theorem to the linear flow $\Psi^{1}$,
	the projection to the projective bundle has a geometric interpretation: It is
	a version of the projection to the Poincar\'{e} sphere, which (in the
	autonomous case) is obtained by attaching a copy of $\mathbb{R}^{d}$ to the
	sphere $\mathbb{S}^{d}$ in $\mathbb{R}^{d+1}$ at the north pole and by taking
	the central projection from the origin in $\mathbb{R}^{d+1}$ to the northern
	hemisphere $\mathbb{S}^{d,+}$ of $\mathbb{S}^{d}$. Then the equator of
	$\mathbb{S}^{d}$ represents infinity. This is closely related to the classical
	construction of the Poincar\'{e} sphere from the global theory of ordinary
	differential equations going back to Poincar\'{e} \cite{Poin}; cf., e.g.,
	Perko \cite[Section 3.10]{Perko}.
	
	The main contributions of this paper are the following: Affine flows on vector
	bundles are lifted to linear flows by multiplying the inhomogeneous term by an
	additional state variable, which is constant. This linear flow on the extended
	state space can be projected to a flow on the projective bundle, where the
	equator can be interpreted as representing the original flow at infinity.
	Selgrade's theorem for linear flows provides a decomposition of the extended
	state space. It turns out that there is a unique Selgrade bundle, whose
	projection is not contained in the equator. We call it the central Selgrade
	bundle. The projections of the other Selgrade bundles are contained in the
	equator, hence we call them the Selgrade bundles at infinity. Since the
	projective flow restricted to the equator is conjugate to the flow of the
	projectivized linear part of the original flow the Selgrade bundles at
	infinity are obtained by the Selgrade bundles of the linear part of the
	original flow. The flow on projective space outside of the equator is
	conjugate to the original affine flow. The projection of the central Selgrade
	bundle contains the image of the chain transitive set of the original affine
	flow. Furthermore, the Morse spectra of the various Selgrade bundles can be
	characterized. The special cases of uniformly hyperbolic and split affine
	systems allow sharper results. For affine control flows generated by affine
	control systems chain controllability properties can be characterized.
	
	The contents of this paper are as follows. In Section \ref{Section2} on
	preliminaries we formulate Selgrade's theorem for linear flows on vector
	bundles and the Morse spectrum after recalling the required notions from the
	topological theory of flows on metric spaces. In Section \ref{Section3} affine
	flows are defined and lifted to linear flows to which Selgrade's theorem is
	applied. Theorem \ref{Theorem_affine1} shows that there is a unique central
	Selgrade bundle and the other Selgrade bundles are \textquotedblleft at
	infinity\textquotedblright. Section \ref{Section4} deduces a formula for the
	central Selgrade bundle of split affine flows, where the homogeneous and the
	inhomogeneous part can be separated, and Section \ref{Section5} describes the
	uniformly hyperbolic case. In Section \ref{Section6} first some notions from
	control theory are introduced, in particular, the correspondence between
	maximal invariant chain transitive sets of the control flow and chain control
	sets is recalled. Then it is proved that chain control sets are unique for
	split affine control systems, the previous results are applied to the affine
	control flows generated by affine control systems, and several examples are presented.
	
	\section{Preliminaries\label{Section2}}
	
	This section collects notation and results for continuous flows on metric
	spaces and recalls Selgrade's theorem for linear flows as well as the Morse spectrum.
	
	\subsection{Flows on metric spaces\label{Subsection2.1}}
	
	For the following concepts for flows on metric spaces cf. Alongi and Nelson
	\cite{AlonN07}, Robinson \cite{Robin98}, and Colonius and Kliemann
	\cite{ColK00, ColK14}.
	
	A flow on a metric space $X$ with metric $d$ is given by a continuous function
	$\Phi:\mathbb{R}\times X\rightarrow X$ satisfying $\Phi(0,x)=x$ and
	$\Phi(t+s,x)=\Phi(t,\Phi(s,x))$ for all $t,s\in\mathbb{R}$ and $x\in X$. Where
	convenient, we also write $\Phi_{t}(x)=\Phi(t,x)$. A conjugacy of flows
	$\Phi^{\prime}$ on $X^{\prime}$ and $\Phi^{\prime\prime}$ on $X^{\prime\prime
	}$ is a homeomorphism $h:X^{\prime}\rightarrow X^{\prime\prime}$ with
	$h(\Phi_{t}^{\prime}(x))=\Phi_{t}^{\prime\prime}(h(x))$ for all $(t,x)\in
	\mathbb{R}\times X^{\prime}$.
	
	For $\varepsilon,\,T>0$ an $(\varepsilon,T)$-chain $\zeta$ for $\Phi$ from $x$
	to $y$ is given by $n\in\mathbb{N}\mathbf{,}$ $T_{0},\ldots,T_{n-1}\geq T$,
	and $x_{0}=x,\ldots,x_{n-1},x_{n}=y\in X$ with $d(\Phi(T_{i},x_{i}%
	),x_{i+1})<\varepsilon$ for $i=0,\ldots,n-1$. For $x\in X$ the $\omega$-limit
	and the $\alpha$-limit set are
	\[
	\omega(x)=\{y\in X|\exists t_{k}\rightarrow\infty:\Phi(t_{k},x)\rightarrow
	y\}\mbox{ and }
	\]%
	\[
	\alpha(x)=\{y\in X\left\vert \exists t_{k}\rightarrow-\infty:\Phi
	(t_{k},x)\rightarrow y\right.  \},
	\]
	respectively. The (forward) chain limit set is
	\[
	\Omega(x)=\{y\in X\left\vert \forall\varepsilon,T>0\,\exists(\varepsilon
	,T)\mbox{-chain from }x\mbox{ to }y\right.  \}.
	\]
	A point $x\in X$ is called chain recurrent if $x\in\Omega(x)$, and a set
	$Y\subset X$ is called chain transitive if $y\in\Omega(x)$ for all $x,y\in Y$.
	Observe that any subset of a chain transitive set is chain transitive, and
	(cf. \cite[Proposition 2.7.10]{AlonN07}) a set is chain transitive if and only
	if its closure is chain transitive. A chain recurrent component is a maximal
	chain transitive set. On a compact metric space these are the connected
	components of the chain recurrent set and the flow restricted to a chain
	recurrent component is chain transitive. If $X$ is chain transitive for a flow
	on $X$, then also the flow is called chain transitive. For a continuous map
	$f:X\rightarrow X$ and $x,y\in X$ an $\varepsilon$-chain from $x$ to $y$ is
	given by $x_{0}=x,x_{1},\ldots,x_{n-1},x_{n}=y$ in $X$ with $d(f(x_{i}%
	),x_{i+1})<\varepsilon$ for all $i$.
	
	The next result is proved in \cite[Theorem 2.7.18]{AlonN07}.
	
	\begin{theorem}
		\label{Theorem_chains}The following properties are equivalent for a flow
		$\Phi$ on a compact metric space $X$ and points $x,y\in X$.
		
		(i) The points $x$ and $y$ satisfy $y\in\Omega(x)$ and $x\in\Omega(y)$.
		
		(ii) For the map $\Phi_{1}:X\rightarrow X$ and every $\varepsilon>0$ there
		exists an $\varepsilon$-chain from $x$ to $y$ and an $\varepsilon$-chain from
		$y$ to $x$.
	\end{theorem}
	
	It immediately follows that the product of two chain transitive flows is chain transitive.
	
	A related concept are Morse decompositions introduced next. Note first that a
	compact subset $K\subset X$ is called isolated invariant for $\Phi$ if the
	following holds: $\Phi_{t}(x)\in K$ for all $x\in K$ and all $t\in\mathbb{R}$
	and there exists a set $N$ with $K\subset\mathrm{int}\,N$, such that $\Phi
	_{t}(x)\in N$ for all $t\in\mathbb{R}$ implies $x\in K$.
	
	\begin{definition}
		\label{defMorsedecomp}A Morse decomposition of a flow $\Phi$ on a compact
		metric space $X$ is a finite collection $\left\{  \mathcal{M}_{i}\left\vert
		i=1,\ldots,\ell\right.  \right\}  $ of nonvoid, pairwise disjoint, and compact
		isolated invariant sets such that
		
		(i) for all $x\in X$ the limit sets satisfy $\omega(x),\,\alpha(x)\subset
		\bigcup_{i=1}^{\ell}\mathcal{M}_{i}$, and
		
		(ii) suppose that there are $\mathcal{M}_{j_{0}},\mathcal{M}_{j_{1}}%
		,\ldots,\mathcal{M}_{j_{n}}$ and $x_{1},\ldots,x_{n}\in X\setminus
		\bigcup_{i=1}^{\ell}\mathcal{M}_{i}$ with $\alpha(x_{i})\subset\mathcal{M}%
		_{j_{i-1}}$ and $\omega(x_{i})\subset\mathcal{M}_{j_{i}}$ for $i=1,\ldots,n$;
		then $\mathcal{M}_{j_{0}}\neq\mathcal{M}_{j_{n}}$.
		
		The elements of a Morse decomposition are called Morse sets. An order is
		defined by the relation $\mathcal{M}_{i}\preceq\mathcal{M}_{j}$ if there are
		indices $j_{0},\ldots,j_{n}$ with $\mathcal{M}_{i}=\mathcal{M}_{j_{0}
		},\mathcal{M}_{j}=\mathcal{M}_{j_{n}}$ and points $x_{j_{i}}\in X$ with
		\[
		\alpha(x_{j_{i}})\subset\mathcal{M}_{j_{i}-1}\mbox{ and }\omega(x_{j_{i}
		})\subset\mathcal{M}_{j_{i}}\mbox{ for }i=1,\ldots,n.
		\]
		
	\end{definition}
	
	We enumerate the Morse sets in such a way that $\mathcal{M}_{i}\,\preceq
	\mathcal{M}_{j}$ implies $i\leq j$. Thus Morse decompositions describe the
	flow as it goes from a lesser (with respect to the order $\preceq$) Morse set
	to a greater Morse set for trajectories that do not start in one of the Morse
	sets. A Morse decomposition $\left\{  \mathcal{M}_{1},\ldots,\mathcal{M}%
	_{\ell}\right\}  $ is called \textit{finer} than a Morse decomposition
	$\left\{  \mathcal{M}_{1}^{\prime},\ldots,\mathcal{M}_{\ell^{\prime}}^{\prime
	}\right\}  $, if for all $j\in\left\{  1,\ldots,\ell^{\prime}\right\}  $ there
	is $i\in\left\{  1,\ldots,\ell\right\}  $ with $\mathcal{M}_{i}\subset
	\mathcal{M}_{j}^{\prime}$.
	
	The following theorem relates chain recurrent components and Morse
	decompositions; cf. \cite[Theorem 8.3.3]{ColK14}.
	
	\begin{theorem}
		For a flow on a compact metric space there exists a finest Morse decomposition
		if and only if the chain recurrent set has only finitely many connected
		components. Then the Morse sets coincide with the chain recurrent components.
	\end{theorem}
	
	\subsection{Linear flows and Selgrade's theorem\label{Subsection2.2}}
	
	We will consider vector bundles $\mathcal{V}=B\times H$, where $B$ is a
	compact metric base space and $H$ is a finite dimensional Hilbert space of
	dimension $d$. A linear flow $\Phi=(\theta,\varphi)$ on $B\times H$ is a flow
	of the form
	\[
	\Phi:\mathbb{R}\times B\times H\rightarrow B\times H,\,\Phi_{t}(b,x)=(\theta
	_{t}b,\varphi(t,b,x))\mbox{ for }(t,b,x)\in\mathbb{R}\times B\times H,
	\]
	where $\theta$ is a flow on the base space $B$ and $\varphi(t,b,x)$ is linear
	in $x$, i.e., $\varphi(t,b,\alpha_{1}x_{1}+\alpha_{2}x_{2})=\alpha_{1}%
	\varphi(t,b,x_{1})+\alpha_{2}\varphi(t,b,x_{2})$ for $\alpha_{1},\alpha_{2}%
	\in\mathbb{R}$ and $x_{1},x_{2}\in H$. We also write $\Phi_{t}(b,\alpha
	_{1}x_{1}+\alpha_{2}x_{2})=\alpha_{1}\Phi_{t}(b,x_{1})+\alpha_{2}\Phi
	_{t}(b,x_{2})$. A closed subset $\mathcal{V}$ of $B\times H$ that intersects
	each fiber $\{b\}\times H,b\in B$, in a linear subspace of constant dimension
	is a subbundle. Let $\mathbb{P}H$ be the projective space for $H$ and denote
	the projection $H\setminus\{0\}\rightarrow\mathbb{P}H$ as well as the
	corresponding map $B\times(H\setminus\{0_{H}\})\rightarrow B\times\mathbb{P}H$
	by the letter $\mathbb{P}$. A linear flow $\Phi$ induces a flow $\mathbb{P}%
	\Phi$ on the projective bundle $B\times\mathbb{P}H$. A metric on $\mathbb{P}H$
	is defined by
	\begin{equation}
		d(p_{1},p_{2})=\min\left\{  \left\Vert \frac{x}{\left\Vert x\right\Vert
		}-\frac{y}{\left\Vert y\right\Vert }\right\Vert ,\left\Vert \frac
		{x}{\left\Vert x\right\Vert }+\frac{y}{\left\Vert y\right\Vert }\right\Vert
		\right\}  \mbox{ for }p_{1}=\mathbb{P}x,p_{2}=\mathbb{P}y.\label{metric_P}%
	\end{equation}
	Then $B\times\mathbb{P}H$ becomes a compact metric space by defining the
	metric as the maximum of the distances in $B$ and $\mathbb{P}H$.
	
	Recall that for a linear flow $\Phi$ two nontrivial invariant subbundles
	$(\mathcal{V}^{+},\mathcal{V}^{-})$ with $B\times H=\mathcal{V}^{+}%
	\oplus\mathcal{V}^{-}$ are exponentially separated if there are $c,\mu>0$
	with
	\begin{equation}
		\left\Vert \Phi_{t}(b,x^{+})\right\Vert \leq ce^{-\mu t}\left\Vert \Phi
		_{t}(b,x^{-})\right\Vert ,t\geq0,\mbox{ for }(b,x^{\pm})\in\mathcal{V}^{\pm
		},\,\left\Vert x^{+}\right\Vert =\left\Vert x^{-}\right\Vert . \label{exp_sep}%
	\end{equation}
	The following is Selgrade's theorem for linear flows; cf. \cite[Theorem
	9.2.5]{ColK14}, and \cite[Theorem 5.1.4]{ColK00} for the result on exponential separation.
	
	\begin{theorem}
		\label{Theorem_Selgrade}Let $\Phi=(\theta,\varphi):\mathbb{R}\times B\times
		H\rightarrow B\times H$ be a linear flow on the vector bundle $B\times H$ with
		chain transitive flow $\theta$ on the base space $B$.\ Then the projected flow
		$\mathbb{P}\Phi$ on $B\times\mathbb{P}H$ has a finite number of chain
		recurrent components $\mathcal{M}_{1},\ldots,\mathcal{M}_{\ell},1\leq\ell\leq
		d=\dim H$. These components form the finest Morse decomposition for
		$\mathbb{P}\Phi$, and they are linearly ordered. The Morse sets will be
		numbered such that $\mathcal{M}_{1}\preceq\cdots\preceq\mathcal{M}_{\ell}$.
		Their lifts
		\[
		\mathcal{V}_{i}=\mathbb{P}^{-1}\mathcal{M}_{i}:=\left\{  (b,x)\in B\times
		H\left\vert x\not =0\Rightarrow(b,\mathbb{P}x)\in\mathcal{M}_{i}\right.
		\right\}  ,
		\]
		are subbundles, called the Selgrade bundles. They form a continuous bundle
		decomposition (a Whitney sum)
		\[
		B\times H=\mathcal{V}_{1}\oplus\cdots\oplus\mathcal{V}_{\ell}.
		\]
		This Selgrade decomposition is the finest decomposition into exponentially
		separated subbundles: For any exponentially separated subbundles
		$(\mathcal{V}^{+},\mathcal{V}^{-})$ there is $1\leq j<\ell$ with
		\[
		\mathcal{V}^{+}=\mathcal{V}_{1}\oplus\cdots\oplus\mathcal{V}_{j}%
		\;\mbox{and\ }\mathcal{V}^{-}=\mathcal{V}_{j+1}\oplus\cdots\oplus
		\mathcal{V}_{\ell}\,.
		\]
		Conversely, subbundles $\mathcal{V}^{+}$ and $\mathcal{V}^{-}$ defined in this
		way are exponentially separated.
	\end{theorem}
	
	\subsection{The Morse spectrum for linear flows\label{Subsection2.4}}
	
	For linear flows $\Phi$ on vector bundles, a number of spectral notions and
	their relations have been considered; cf., e.g., Sacker and Sell \cite{SacS},
	Johnson, Palmer, and Sell \cite{JoPS87}, Kawan and Stender \cite{KawS12}. An
	appropriate spectral notion in the present context is provided by the Morse
	spectrum defined as follows; cf. Colonius and Kliemann \cite{ColK14} and Alves
	and San Martin \cite{AlvSM16}, and for generalizations cf. Gr\"{u}ne
	\cite{Grue00} and Patr\~{a}o and San Martin \cite{PatSM07}.
	
	For $\varepsilon,T>0$ let an $(\varepsilon,T)$-chain $\zeta$ of $\mathbb{P}%
	\Phi$ be given by $n\in\mathbb{N}\mathbf{,}$ $T_{0},\ldots,T_{n-1}\geq T$, and
	$(b_{0},p_{0}),\ldots,(b_{n},p_{n})\in B\times\mathbb{P}H$ with $d(\mathbb{P}%
	\Phi(T_{i},b_{i},p_{i}),(b_{i+1},p_{i+1}))<\varepsilon$ for $i=0,\ldots,n-1$.
	With total time $\tau=\sum_{i=0}^{n-1}T_{i}$ let the exponential growth rate
	of $\zeta$ be
	\[
	\lambda(\zeta):=\frac{1}{\tau}\left(  \sum_{i=0}^{n-1}\log\left\Vert
	\Phi(T_{i},b_{i},x_{i})\right\Vert -\log\left\Vert (b_{i},x_{i})\right\Vert
	\right)  \mbox{ with }\mathbb{P}x_{i}=\,p_{i}.
	\]
	Define the Morse spectrum of a subbundle $\mathcal{V}_{i}=\mathbb{P}%
	^{-1}\mathcal{M}_{i}$ generated by $\mathcal{M}_{i}$ as
	\[
	\mathbf{\Sigma}_{Mo}(\mathcal{V}_{i};\Phi)=\left\{
	\begin{array}
		[c]{c}%
		\lambda\in\mathbb{R}:\mbox{\ there\ are}\;\varepsilon^{k}\rightarrow
		0,\,T^{k}\rightarrow\infty\;\,\mbox{and}\,\;\\
		(\varepsilon^{k},\,T^{k})\mbox{-chains}\mathrm{\;}\zeta^{k}%
		\;\mbox{in}\;\mathcal{M}_{i}\;\mbox{with}\mathrm{\;}\lambda(\zeta
		^{k})\rightarrow\lambda\;\mbox{as}\;k\rightarrow\infty
	\end{array}
	\right\}  .
	\]
	The Morse spectrum has the following properties; cf. \cite[Theorem 9.3.5 and
	Theorem 9.4.1]{ColK14}
	
	\begin{theorem}
		\label{Theorem_Morse}For a linear flow $\Phi$ on a vector bundle $B\times H$
		with chain transitive base space $B$ the Morse spectrum $\mathbf{\Sigma}%
		_{Mo}(\mathcal{V}_{i};\Phi)$ of a Selgrade bundle $\mathcal{V}_{i}$ is a
		compact interval, and for every $(b,x)\in B\times(H\setminus\{0_{H}\})$ the
		Lyapunov exponent $\lambda(b,x)=\lim\sup_{t\rightarrow\infty}\frac{1}{t}%
		\log\left\Vert \varphi(t,b,x)\right\Vert $ is contained in some
		$\mathbf{\Sigma}_{Mo}(\mathcal{V}_{i};\Phi)$.
	\end{theorem}
	
	The spectral intervals $\mathbf{\Sigma}_{Mo}(\mathcal{V}_{i};\Phi)$ need not
	be disjoint. In particular, there may exist two \textquotedblleft
	center\textquotedblright\ subbundles with $0$ in the Morse spectrum; cf.
	Salamon and Zehnder \cite[Example 2.14]{SalZ88} and also Example
	\ref{Example6.8}.
	
	\section{Selgrade's theorem for affine flows and the Poincar\'{e}
		sphere\label{Section3}}
	
	In this section, affine flows are lifted to linear flows on an augmented state
	space and the Selgrade decomposition on this space is analyzed.
	
	The following construction of affine flows is taken from Colonius and Santana
	\cite{ColSan11}.
	
	\begin{definition}
		\label{Definition_affine}Let $B\times H$ be a vector bundle with compact
		metric base space $B$. A continuous map $\Psi=(\theta,\psi):\mathbb{R}\times
		B\times H\rightarrow B\times H$ is called an affine flow on $B\times H$ if
		there are a linear flow $\Phi=(\theta,\varphi)$ and a function $f:B\rightarrow
		L^{\infty}(\mathbb{R},H)$ such that $f$ satisfies
		\begin{equation}
			f(b)(t+s)=f(\theta_{s}b)(t)\mbox{ for all }b\in B\mbox{ and almost all }t,s\in
			\mathbb{R}, \label{CPAA3}%
		\end{equation}
		and for all $(t,b,x)\in\mathbb{R}\times B\times H$%
		\begin{equation}
			\Psi_{t}(b,x)=(\theta_{t}b,\psi(t,b,x))=(\theta_{t}b,\varphi(t,b,x)+\int
			_{0}^{t}\varphi(t-s,\theta_{s}b,f(b)(s))ds). \label{CPAA4a}%
		\end{equation}
		
	\end{definition}
	
	The base flows of $\Psi$ and $\Phi$ coincide and the integral in
	(\ref{CPAA4a}) is a Lebesgue integral in the $H$-component. The flow property
	of $\Psi$ is expressed by the cocycle property $\psi(t+s,b,x)=\psi
	(t,\theta_{s}b,\psi(s,b,x))$, which follows from (\ref{CPAA3}). With
	$f(b,s):=f(b)(s),s\in\mathbb{R}$, formula (\ref{CPAA4a}) can be written in the
	more concise form%
	\begin{equation}
		\Psi_{t}(b,x)=\Phi_{t}(b,x)+\int_{0}^{t}\Phi_{t-s}(\theta_{s}b,f(b,s))ds.
		\label{CPAA4}%
	\end{equation}
	We will always assume that the base flow $\theta$ on $B$ is chain transitive.
	Next we formulate a simple but fundamental construction for the present paper.
	
	\begin{proposition}
		\label{Proposition_lift}Any affine flow $\Psi=(\theta,\psi)$ on $B\times H$
		can be lifted to a linear flow $\Psi^{1}$ on $B\times H^{1}$, $H^{1}%
		:=H\times\mathbb{R}$, by defining for $(t,b,x,r)\in\mathbb{R}\times\left(
		B\times H\times\mathbb{R}\right)  $,
		\[
		\Psi_{t}^{1}(b,x,r)=\left(  \theta_{t}b,\psi^{1}(t,b,x,r\right)  ,r)=\left(
		\Phi_{t}(b,x)+r\int_{0}^{t}\Phi_{t-s}(\theta_{s}b,f(b,s))ds,r\right)  .
		\]
		
	\end{proposition}
	
	\begin{proof}
		Continuity and the flow properties are obvious. We prove linearity. For
		$\alpha,\beta\in\mathbb{R}$ and $(b,x,r),\allowbreak(b,y,s)\in B\times
		H\times\mathbb{R}$ one has
		\[%
		\begin{array}
			[c]{l}%
			\Psi_{t}^{1}(b,\alpha(x,r)+\beta(y,s))\\
			=\left(  \Phi_{t}(b,\alpha x+\beta y)+(\alpha r+\beta s)\int_{0}^{t}%
			\Phi_{t-\sigma}(\theta_{\sigma}b,f(b,\sigma))d\sigma,\alpha r+\beta s\right)
			\\
			=(\alpha(\Phi_{t}(b,x)+r\int_{0}^{t}\Phi_{t-\sigma}(\theta_{\sigma
			}b,f(b,\sigma))d\sigma\\
			+\,\,\beta(\Phi_{t}(b,y)+s\int_{0}^{t}\Phi_{t-\sigma}(\theta_{\sigma
			}b,f(b,\sigma))d\sigma),\alpha r+\beta s)\\
			=\alpha\Psi_{t}^{1}(b,x,r)+\beta\Psi_{t}^{1}(b,y,s).
		\end{array}
		\]
		
	\end{proof}
	
	We will apply Selgrade's theorem to the linear flow $\Psi^{1}$. Define subsets
	of $H^{1}$ by $H^{1,0}=H\times\{0\}$ and $\,H^{1,1}=H\times\left(
	\mathbb{R}\setminus\{0\}\right)  $. One obtains subsets of $\mathbb{P}H^{1}$
	given by
	\[
	\mathbb{P}H^{1,0}=\{\mathbb{P}(x,0)\in\mathbb{P}H^{1}\left\vert x\in H\right.
	\},\,\,\mathbb{P}H^{1,1}=\{\mathbb{P}(x,r)\in\mathbb{P}H^{1}\left\vert x\in
	H,r\not =0\right.  \}.
	\]
	Note that $\mathbb{P}H^{1,1}=\mathbb{P}\left(  H\times\{1\}\right)  $. The
	projective space $\mathbb{P}H^{1}=\overline{\mathbb{P}H^{1,1}}$ is the
	disjoint union of these subsets, the set $\mathbb{P}H^{1,0}$ is closed and the
	set $\mathbb{P}H^{1,1}$ is open. For the unit sphere $\mathbb{S}H^{1}$ of
	$H^{1}$ denote the northern hemisphere and the equator by $\mathbb{S}^{+}%
	H^{1}:=\left\{  (x,r)\in\mathbb{S}H^{1}\left\vert x\in H,r>0\right.  \right\}
	$ and $\mathbb{S}^{0}H^{1}=\{(x,0)\in\mathbb{S}H^{1}\left\vert x\in H\right.
	\}$, respectively. Note that $\mathbb{P}H^{1,1}$ can be identified with the
	northern hemisphere $\mathbb{S}^{+}H^{1}$.
	
	\begin{definition}
		The Poincar\'{e} sphere bundle is given by $B\times\mathbb{S}H^{1}$ and the
		projective Poincar\'{e} bundle is $B\times\mathbb{P}H^{1}$.
	\end{definition}
	
	The linear flow $\Psi^{1}$ on $B\times H^{1}$ induces a flow $\mathbb{P}
	\Psi^{1}$ on the projective bundle $B\times\mathbb{P}H^{1}$. It can be
	restricted to $B\times\mathbb{P}H^{1,i},i=0,1$, since under the flow $\Psi
	^{1}$ the last component remains fixed. The following proposition shows that
	$\mathbb{P}\Psi^{1}$ restricted to $B\times\mathbb{P}H^{1,0}$ is conjugate to
	the flow induced by the linear part $\Phi$ of $\Psi$ on $B\times\mathbb{P}H$,
	and that the flow $\Psi$ on $B\times H$ is conjugate to the flow
	$\mathbb{P}\Psi^{1}$ restricted to $B\times\mathbb{P}H^{1,1}$.
	
	\begin{proposition}
		\label{Proposition_e}(i) For $(b,x)\in B\times H$ the equality $\Psi_{t}
		^{1}(b,x,0)=(\Phi_{t}(b,x),0),\allowbreak t\in\mathbb{R}$, holds, and the
		projective map
		\[
		h^{0}:B\times\mathbb{P}H\rightarrow B\times\mathbb{P}H^{1,0},h^{0}
		(b,\mathbb{P}x)=(b,\mathbb{P}(x,0)),
		\]
		is a conjugacy of the flows $\mathbb{P}\Phi$ and $\mathbb{P}\Psi^{1}$
		restricted to $B\times\mathbb{P}H^{1,0}$. In particular, the chain recurrent
		components $\mathcal{M}_{i}=\mathbb{P}\mathcal{V}_{i}$ of $\mathbb{P}\Phi$
		yield the chain recurrent components $h^{0}(\mathcal{M}_{i})=\mathbb{P}\left(
		\mathcal{V}_{i}\times\{0\}\right)  ,i\in\{1,\ldots,\ell\}$, of $\mathbb{P}
		\Psi^{1}$ restricted to $B\times\mathbb{P}H^{1,0}$, and their order is preserved.
		
		(ii) The map
		\[
		h^{1}:B\times H\rightarrow B\times\mathbb{P}H^{1,1},(b,x)\mapsto
		\mathbb{P}(b,x,1)=(b,\mathbb{P}(x,1)),
		\]
		is a conjugacy of the flows $\Psi$ on $B\times H$ and $\mathbb{P}\Psi^{1}$
		restricted to $B\times\mathbb{P}H^{1,1}$.
		
		(iii) For $\varepsilon,T>0$ any $(\varepsilon,T)$-chain in $B\times H$ is
		mapped by $h^{1}$ onto a $(2\varepsilon,T)$-chain in $B\times\mathbb{P}%
		H^{1,1}$, hence any chain transitive set $C\subset B\times H$ is mapped onto a
		chain transitive set $h^{1}(C)\subset B\times\mathbb{P}H^{1,1}$.
		
		(iv) For a subset $C\subset B\times H$ the set $\{x\in H\left\vert (b,x)\in
		C\mbox{ for some }b\in B\right.  \}$ is bounded if and only if $\overline
		{h^{1}(C)}\cap(B\times\mathbb{P}H^{1,0})=\varnothing$.
	\end{proposition}
	
	\begin{proof}
		(i), (ii) The first assertion in (i) is clear by the definition of $\Psi^{1}$.
		Recall that $\mathbb{P}H=(H\setminus\{0_{H}\})/\thicksim$, where $\thicksim$ is
		the equivalence relation $x\thicksim y$ if $y=\lambda x$ with some
		$\lambda\not =0$. Given a basis of $H$ an atlas of $\mathbb{P}H$ is given by
		$n$ charts $(U_{i},\alpha_{i})$, where $U_{i}$ is the set of equivalence
		classes $[x_{1}:\cdots:x_{d}]$ with $x_{i}\not =0$ (using homogeneous
		coordinates) and $\alpha_{i}:U_{i}\rightarrow\mathbb{R}^{d-1}$ is defined by
		\[
		\alpha_{i}([x_{1}:\cdots:x_{d}])=\left(  \frac{x_{1}}{x_{i}},\ldots
		,\widehat{\frac{x_{i}}{x_{i}}},\ldots,\frac{x_{d}}{x_{i}}\right)  ;
		\]
		here the hat means that the $i$-th entry is omitted. In homogeneous
		coordinates, the levels $\mathbb{P}H^{1,i}$ are described by
		\[
		\mathbb{P}H^{1,i}=\left\{  [x_{1}:\cdots:x_{d}:i]\left\vert (x_{1}%
		,\ldots,x_{d})\in\mathbb{R}^{d}\right.  \right\}  \mbox{ for }i=0,1.
		\]
		Observe that, by homogeneity, $\mathbb{P}H^{1,0}=\left\{  [x_{1}:\cdots
		:x_{d}:0]\left\vert \mbox{ }\left\Vert (x_{1},\ldots,x_{d})\right\Vert
		=1\right.  \right\}  $. Any trajectory of $\mathbb{P}\Psi^{1}$ is obtained as
		the projection of a trajectory of $\Psi^{1}$ with initial condition satisfying
		$r^{0}=0$ or $1$, since $[x_{1}^{0}:\cdots:x_{d}^{0}:r^{0}]=[\frac{x_{1}^{0}%
		}{r^{0}}:\cdots:\frac{x_{d}^{0}}{r^{0}}:1]$ for $r^{0}\not =0$. A trivial
		atlas for $\mathbb{P}H^{1,1}$ is given by $\left\{  (U_{d+1},\alpha
		_{d+1})\right\}  $ proving that $\mathbb{P}H^{1,1}$ is a manifold which is
		diffeomorphic to $\mathbb{R}^{d}$. Observe also that $\mathbb{P}H^{1,0}$ is
		diffeomorphic to $\mathbb{P}^{d-1}$.
		
		In homogeneous coordinates the spaces $\mathbb{P}H$ and $\mathbb{P}H^{1,0}$
		are diffeomorphic under the map associating to $[x_{1}:\cdots:x_{d}]$ the
		value $[x_{1}:\cdots:x_{d}:0]$. For any trajectory $(\theta_{t}b,\psi
		^{1}(t,b,x^{0},r),r)$ of $\Psi^{1}$ in $B\times H^{1,1}$, the projection to
		$\mathbb{P}H^{1,1}\subset\mathbb{P}H^{1}$ is $(\theta_{t}b,[\psi_{1}
		^{1}(t,b,x^{0},r):\cdots:\psi_{d}^{1}(t,b,x^{0},r):r])$, where $\psi_{i}
		^{1}(t,b,x^{0},r)$ is the $i$-th component of $\psi^{1}(t,b,x^{0},r)$. Now the
		conjugacy properties in (i) and (ii) follow. The assertion in (i) on the chain
		recurrent components holds, since the state spaces are compact.
		
		(iii) In view of assertion (ii) it suffices to show that $d((b,x),(b^{\prime
		},x^{\prime}))<\varepsilon$ in $B\times H$ implies $d(h^{1}(b,x),h^{1}%
		(b^{\prime},x^{\prime}))<2\varepsilon$ in $B\times\mathbb{P}H^{1}$. Here\ the
		metric in $\mathbb{P}H$ is defined in (\ref{metric_P}).\textbf{ }Since
		$d(b,b^{\prime})<\varepsilon$ it suffices to estimate the components in the
		Poincar\'{e} sphere $\mathbb{P}H^{1}$. For the projections to $\mathbb{S}%
		H^{1}$ we obtain
		\[
		\left\Vert \frac{(x,1)}{\left\Vert (x,1)\right\Vert }-\frac{(x^{\prime}%
			,1)}{\left\Vert (x^{\prime},1)\right\Vert }\right\Vert =\frac{\left\Vert
			\left(  \left\Vert (x^{\prime},1)\right\Vert x-\left\Vert (x,1)\right\Vert
			x^{\prime},\left\Vert (x^{\prime},1)\right\Vert -\left\Vert (x,1)\right\Vert
			\right)  \right\Vert }{\left\Vert (x,1)\right\Vert \left\Vert (x^{\prime
			},1)\right\Vert }.
		\]
		Observe that $\left\Vert x\right\Vert -\left\Vert x^{\prime}\right\Vert
		\leq\left\Vert x-x^{\prime}\right\Vert <\varepsilon$ and $\left\Vert
		(x,1)\right\Vert -\left\Vert (x^{\prime},1)\right\Vert <\varepsilon$. Thus the
		last component satisfies%
		\[
		\frac{\left\Vert (x^{\prime},1)\right\Vert -\left\Vert (x,1)\right\Vert
		}{\left\Vert (x,1)\right\Vert \left\Vert (x^{\prime},1)\right\Vert
		}<\varepsilon.
		\]
		Concerning the other components we find $\delta(\varepsilon)$ with $\left\vert
		\delta(\varepsilon)\right\vert <\varepsilon$ such that $\left\Vert (x^{\prime
		},1)\right\Vert =\left\Vert (x,1)\right\Vert +\delta(\varepsilon)$. Hence%
		\[
		\left\Vert \left\Vert (x^{\prime},1)\right\Vert x-\left\Vert (x,1)\right\Vert
		x^{\prime}\right\Vert \leq\left\Vert (x,1)\right\Vert \left\Vert x-x^{\prime
		}\right\Vert +\delta(\varepsilon)\left\Vert x\right\Vert <\left\Vert
		(x,1)\right\Vert \varepsilon+\delta(\varepsilon)\left\Vert x\right\Vert
		\]
		implying%
		\[
		\frac{\left\Vert \left\Vert (x^{\prime},1)\right\Vert x-\left\Vert
			(x,1)\right\Vert x^{\prime}\right\Vert }{\left\Vert (x,1)\right\Vert
			\left\Vert (x^{\prime},1)\right\Vert }\leq\frac{\left\Vert (x,1)\right\Vert
			\varepsilon+\delta(\varepsilon)\left\Vert x\right\Vert }{\left\Vert
			(x,1)\right\Vert \left\Vert (x^{\prime},1)\right\Vert }<\varepsilon
		+\delta(\varepsilon)<2\varepsilon.
		\]

		(iv) Consider a sequence $(b^{n},x^{n}),n\in\mathbb{N}$, in $C$. For the
		images $h^{1}(b^{n},x^{n})=(b^{n},\mathbb{P}(x^{n},1))$ the points
		$\mathbb{P}(x^{n},1)$ have homogeneous coordinates satisfying
		\[
		\lbrack x_{1}^{n}:\cdots:x_{d}^{n}:1]=\left[  \frac{x_{1}^{n}}{\left\Vert
			x^{n}\right\Vert }:\cdots:\frac{x_{d}^{n}}{\left\Vert x^{n}\right\Vert }%
		:\frac{1}{\left\Vert x^{n}\right\Vert }\right]  .
		\]
		Then $\left\Vert x^{n}\right\Vert \rightarrow\infty$ if and only if $\frac
		{1}{\left\Vert x^{n}\right\Vert }\rightarrow0$ for $n\rightarrow\infty$
		meaning that the distance of $(b^{n},\mathbb{P}(x^{n},1))$ to $B\times
		\mathbb{P}H^{1,0}$ converges to $0$.
	\end{proof}
	
	Observe that chain transitivity of $h^{1}(C)$ for $C\subset B\times H$ implies
	chain transitivity of the closure $\overline{h^{1}(C)}\subset B\times
	\mathbb{P}H^{1}$.
	
	The Selgrade decomposition provided by Theorem \ref{Theorem_Selgrade} can be
	used for the linear flow $\Psi^{1}$ on $B\times H^{1}$. We obtain
	\begin{equation}
		B\times H^{1}=\mathcal{V}_{1}^{1}\oplus\cdots\oplus\mathcal{V}_{\ell^{1}}%
		^{1}\mbox{ with }1\leq\ell^{1}\leq d+1\mbox{ and }\sum\nolimits_{j=1}%
		^{\ell^{1}}\dim\mathcal{V}_{j}^{1}=d+1, \label{Sel_1}%
	\end{equation}
	and let $\mathcal{M}_{j}^{1}:=\mathbb{P}\mathcal{V}_{j}^{1},j\in
	\{1,\ldots,\ell^{1}\}$, be the associated chain recurrent components of
	$\mathbb{P}\Psi^{1}$ on $B\times\mathbb{P}H^{1}$. Furthermore, $\mathcal{M}%
	_{i}:=\mathbb{P}\mathcal{V}_{i}\subset B\times\mathbb{P}H$ denotes the chain
	recurrent component corresponding to a Selgrade bundle $\mathcal{V}_{i}$ of
	the linear part $\Phi$ of $\Psi$.
	
	Note that a Selgrade bundle $\mathcal{V}_{j}^{1}$ of $\Psi^{1}$ satisfies
	$\mathcal{V}_{j}^{1}\cap\left(  B\times H^{1,1}\right)  \not =\varnothing$ if
	and only if there is $(b,x,r)\in\mathcal{V}_{j}^{1}$ with $r\not =0$ and this
	is equivalent to $\mathcal{M}_{j}^{1}\cap\left(  B\times\mathbb{P}%
	H^{1,1}\right)  \not =\varnothing$. Furthermore, a Selgrade bundle satisfies
	$\mathcal{V}_{j}^{1}\cap\left(  B\times H^{1,0}\right)  =B\times\{\left(  0_{H},0\right)  \}$ if
	and only if $\mathcal{M}_{j}^{1}\subset B\times\mathbb{P}H^{1,1}$.
	
	The detailed description of the Selgrade bundles $\mathcal{V}_{j}^{1}$ of
	$\Psi^{1}$ will be based on dimension arguments. We prepare this analysis by
	the following lemma discussing the relations between the subbundles
	$\mathcal{V}_{i}\times\{0\}$ and the Selgrade bundles $\mathcal{V}_{j}^{1}$.
	
	\begin{lemma}
		\label{Lemma_dim1}(i) For every $i\in\{1,\ldots,\ell\}$ there is
		$j(i)\in\{1,\ldots,\ell^{1}\}$ with $\mathbb{P}\left(  \mathcal{V}_{i}%
		\times\{0\}\right)  \allowbreak\subset\mathcal{M}_{j(i)}^{1}$ and
		$\mathcal{V}_{i}\times\{0\}\subset\mathcal{V}_{j(i)}^{1}$.
		
		(ii) A subbundle $\mathcal{V}_{i}\times\{0\},i\in\{1,\ldots,\ell\}$, is a
		proper subset of the Selgrade bundle $\mathcal{V}_{j}^{1}$ containing it if
		and only if
		\begin{equation}
			\dim\mathcal{V}_{j}^{1}>\sum\nolimits_{k\in I(j)}\dim\left(  \mathcal{V}%
			_{k}\times\{0\}\right)  =\sum\nolimits_{k\in I(j)}\dim\mathcal{V}_{k},
			\label{dim4}%
		\end{equation}
		where $I(j)$ is the set of all indices $k$ with $\mathcal{V}_{k}%
		\times\{0\}\subset\mathcal{V}_{j}^{1}$.
	\end{lemma}
	
	\begin{proof}
		(i) The Selgrade decomposition for $\Phi$ yields that the projections
		$\mathcal{M}_{i}=\mathbb{P}\mathcal{V}_{i}$ to $B\times\mathbb{P}H$ are the
		chain recurrent components of $\mathbb{P}\Phi$. By Proposition
		\ref{Proposition_e}(i) it follows that $h^{0}(\mathcal{M}_{i})=\mathbb{P}%
		\left(  \mathcal{V}_{i}\times\{0\}\right)  $ is a chain recurrent component of
		$\mathbb{P}\Psi^{1}$ restricted to $B\times\mathbb{P}H^{1,0}\subset
		B\times\mathbb{P}H^{1}$. Hence $\mathbb{P}\left(  \mathcal{V}_{i}%
		\times\{0\}\right)  $ is chain transitive for $\mathbb{P}\Psi^{1}$. Thus for
		every $i\in\{1,\ldots,\ell\}$ there is $j$ with $\mathbb{P}\left(
		\mathcal{V}_{i}\times\{0\}\right)  \subset\mathcal{M}_{j}^{1}$ and
		$\mathcal{V}_{i}\times\{0\}\subset\mathcal{V}_{j}^{1}$.
		
		(ii) The inequality $\dim\mathcal{V}_{j}^{1}\geq\sum\nolimits_{k\in I(j)}%
		\dim\left(  \mathcal{V}_{k}\times\{0\}\right)  $ holds, since the sum of the
		subbundles $\mathcal{V}_{k}\times\{0\},k\in I(j)$, is direct, and equality
		holds if and only if $\mathcal{V}_{j}^{1}=\bigoplus_{k\in I(j)}\left(
		\mathcal{V}_{k}\times\{0\}\right)  $.
		
		Suppose that $\mathcal{V}_{i}\times\{0\}$ is a proper subset of $\mathcal{V}%
		_{j}^{1}$. Since by Proposition \ref{Proposition_e}(i) the sets $\mathcal{M}%
		_{k}\times\{0\}$ are chain recurrent components of $\mathbb{P}\Psi^{1}$
		restricted to $B\times H^{1,0}$ it follows that there exists $(b,x,r)\in
		\mathcal{V}_{j}^{1}$ with $r\not =0$. Thus $\bigoplus\nolimits_{k\in
			I(j)}\left(  \mathcal{V}_{k}\times\{0\}\right)  \subset\mathcal{V}_{j}^{1}$ is
		a proper inclusion implying (\ref{dim4}). Conversely, suppose that
		(\ref{dim4}) holds. If $\left\vert I(j)\right\vert >1$ it follows trivially
		that $\mathcal{V}_{i}\times\{0\},i\in I(j)$, is a proper subset of
		$\mathcal{V}_{j(i)}^{1}$. If there is a single $\mathcal{V}_{k}\times
		\{0\}\subset\mathcal{V}_{j}^{1}$ the inequality $\dim\mathcal{V}_{j}^{1}%
		>\dim\mathcal{V}_{k}$ implies that the inclusion $\mathcal{V}_{k}%
		\times\{0\}\subset\mathcal{V}_{j}^{1}$ is proper.
	\end{proof}
	
	The following lemma contains basic information on the Selgrade bundles of
	$\Psi^{1}$.
	
	\begin{lemma}
		\label{Lemma_affine2}There exists a unique Selgrade bundle $\mathcal{V}%
		_{j}^{1}$ of $\Psi^{1}$ such that $\mathcal{V}_{j}^{1}\cap\left(  B\times
		H^{1,1}\right)  \not =\varnothing$. The dimension of $\mathcal{V}_{j}^{1}$ is
		given by
		\begin{equation}
			\dim\mathcal{V}_{j}^{1}=1+\sum\nolimits_{i}\dim\mathcal{V}_{i}, \label{dim0}%
		\end{equation}
		where the summation is over all $i\in\{1,\ldots,\ell\}$ such that
		$\mathcal{V}_{i}\times\{0\}\subset\mathcal{V}_{j}^{1}$. The other Selgrade
		bundles of $\Psi^{1}$ are the subbundles $\mathcal{V}_{i}\times\{0\}$ which
		are not contained in $\mathcal{V}_{j}^{1}$.
	\end{lemma}
	
	\begin{proof}
		Due to the decomposition (\ref{Sel_1}) there is at least one Selgrade bundle
		$\mathcal{V}_{j}^{1}$ with $\mathcal{V}_{j}^{1}\cap(B\times H^{1,1}%
		)\not =\varnothing$ or, equivalently, $\mathcal{M}_{j}^{1}\cap\left(
		B\times\mathbb{P}H^{1,1}\right)  \not =\varnothing$. By Lemma \ref{Lemma_dim1}%
		(i) the projections $\mathbb{P}\left(  \mathcal{V}_{i}\times\{0\}\right)
		,i\in\{1,\ldots,\ell\}$, are chain transitive for $\mathbb{P}\Psi^{1}$. Let
		$\mathcal{M}_{j}^{1},j\in J$, be the chain recurrent components of
		$\mathbb{P}\Psi^{1}$ with $\mathcal{M}_{j}^{1}\cap\left(  B\times
		\mathbb{P}H^{1,1}\right)  \not =\varnothing$ and containing some set
		$\mathbb{P}\left(  \mathcal{V}_{i}\times\{0\}\right)  $, and let $I$ be the
		set of all $i\in\{1,\ldots,\ell\}$ such that $\mathcal{V}_{i}\times\{0\}$ is
		contained in some $\mathcal{V}_{j}^{1},j\in J$.
		
		\textbf{Case 1:} Suppose that $J\not =\varnothing$. Certainly $\mathcal{V}%
		_{i}\times\{0\},i\in I$, is a proper subset of the Selgrade bundle
		$\mathcal{V}_{j}^{1}$ containing it. Applying Lemma \ref{Lemma_dim1}(ii) for
		every $j\in J$ one finds that
		\begin{equation}
			\sum\nolimits_{j\in J}\dim\mathcal{V}_{j}^{1}\geq\left\vert J\right\vert
			+\sum\nolimits_{i\in I}\dim\mathcal{V}_{i}. \label{3.5}%
		\end{equation}
		By Lemma \ref{Lemma_dim1}(i) also the sets $\mathbb{P}\left(  \mathcal{V}%
		_{i}\times\{0\}\right)  ,i\in\{1,\ldots,\ell\}\setminus I$, are contained in
		some chain recurrent component $\mathcal{M}_{j(i)}^{1}$ of $\mathbb{P}\Psi
		^{1}$. Using (\ref{3.5}) we get
		\begin{align}
			d+1  &  \geq\sum\nolimits_{j\in J}\dim\mathcal{V}_{j}^{1}+\sum\nolimits_{j\in
				\{1,\ldots,\ell^{1}\}\setminus J}\dim\mathcal{V}_{j}^{1}\nonumber\\
			&  \geq\left\vert J\right\vert +\sum\nolimits_{i\in I}\dim(\mathcal{V}%
			_{i}\times\{0\})+\sum\nolimits_{i\in\{1,\ldots,\ell\}\setminus I}\dim\left(
			\mathcal{V}_{i}\times\{0\}\right) \label{dim3}\\
			&  =\left\vert J\right\vert +\sum\nolimits_{i=1}^{\ell}\dim\mathcal{V}%
			_{i}=\left\vert J\right\vert +d.\nonumber
		\end{align}
		Since $\left\vert J\right\vert \geq1$ here equalities hold and $\left\vert
		J\right\vert =1$. In particular, there is a unique Selgrade bundle
		$\mathcal{V}_{j}^{1}$ containing some $\mathcal{V}_{i}\times\{0\}$ and these
		are the subbundles with index $i\in I$. Furthermore, one obtains%
		\begin{equation}
			\dim\mathcal{V}_{j}^{1}+\sum\nolimits_{j\in\{1,\ldots,\ell^{1}\}\setminus
				J}\dim\mathcal{V}_{j}^{1}=\dim\mathcal{V}_{j}^{1}+\sum\nolimits_{i\in
				\{1,\ldots,\ell\}\setminus I}\dim\left(  \mathcal{V}_{i}\times\{0\}\right)  .
			\label{dim5}%
		\end{equation}
		If there is $i\in\{1,\ldots,\ell\}\setminus I$ such that $\mathbb{P}\left(
		\mathcal{V}_{i}\times\{0\}\right)  $ is properly contained in a chain
		recurrent component $\mathcal{M}_{j(i)}^{1}$ with $j(i)\not \in J$, then Lemma
		\ref{Lemma_dim1}(ii) implies that $\dim\mathcal{V}_{j(i)}^{1}\geq1+\dim\left(
		\mathcal{V}_{i}\times\{0\}\right)  $. This yields a contradiction to
		(\ref{dim5}) and shows that $\mathcal{V}_{i}\times\{0\}$ is a Selgrade bundle
		for all $i\in\{1,\ldots,\ell\}\setminus I$.
		
		We conclude that the Selgrade bundles of $\Psi^{1}$ are given by
		$\mathcal{V}_{j}^{1}$ and the subbundles $\mathcal{V}_{i}\times\{0\}$ which
		are not contained in $\mathcal{V}_{j}^{1}$. This proves the assertion in case 1.
		
		\textbf{Case 2:} Suppose that $J=\varnothing$, i.e., the subbundles
		$\mathcal{V}_{j}^{1}$ with $\mathcal{M}_{j}^{1}\cap\left(  B\times
		\mathbb{P}H^{1,1}\right)  \not =\varnothing$ do not contain any $\mathcal{V}%
		_{i}\times\{0\}$. Now define $J_{1}$ as the set of indices with $\mathcal{M}%
		_{j}^{1}\cap\left(  B\times\mathbb{P}H^{1,1}\right)  \not =\varnothing$ and
		note that $\left\vert J_{1}\right\vert \geq1$. Since $\mathcal{V}_{j}^{1}%
		\cap\left(  \mathcal{V}_{i}\times\{0\}\right)  =B\times\{(0,0)\}$ for all
		$j\in J_{1}$ and all $i\in\{1,\ldots,\ell\}$ Lemma \ref{Lemma_dim1}(i) implies
		that
		\[
		d+1\geq\sum\nolimits_{j\in J_{1}}\dim\mathcal{V}_{j}^{1}+\sum\nolimits_{i=1}%
		^{\ell}\dim\mathcal{V}_{i}=\sum\nolimits_{j\in J_{1}}\dim\mathcal{V}_{j}%
		^{1}+d.
		\]
		It follows that equality holds here and $\left\vert J_{1}\right\vert =1$, thus
		there is a unique Selgrade bundle $\mathcal{V}_{j}^{1}$ with $\mathcal{M}%
		_{j}^{1}\cap\left(  B\times\mathbb{P}H^{1,1}\right)  \not =\varnothing$ and
		$\dim V_{j}^{1}=1$. By Lemma \ref{Lemma_dim1}(i) every set $\mathbb{P}\left(
		\mathcal{V}_{i}\times\{0\}\right)  $ is contained in some chain recurrent
		component $\mathcal{M}_{j(i)}^{1}$ of $\mathbb{P}\Psi^{1}$. Let $J_{2}$ be the
		set of all Selgrade bundles containing some $\mathcal{V}_{i}\times\{0\}$. If
		there is a subbundle $\mathcal{V}_{i}\times\{0\}$ which is a proper subset of
		$\mathcal{V}_{j(i)}^{1}$ Lemma \ref{Lemma_dim1}(ii) implies the contradiction%
		\[
		d+1\geq1+\sum\nolimits_{j\in J_{2}}\dim\mathcal{V}_{j}^{1}>1+\sum
		\nolimits_{i\in\{1,\ldots,\ell\}}\dim\mathcal{V}_{i}=1+d.
		\]
		We conclude that, in addition to $\mathcal{V}_{j}^{1}$, all subbundles
		$\mathcal{V}_{i}\times\{0\},i\in\{1,\ldots,\ell\}$, are Selgrade bundles of
		$\Psi^{1}$. This proves the assertion in case 2.
	\end{proof}
	
	Proposition \ref{Proposition_e}(iv) shows that $B\times\mathbb{P}H^{1,0}$ may
	be interpreted as a representation of $B\times H$ at infinity. This motivates
	us to call subbundle at infinity any subbundle of the form $\mathcal{V}%
	_{i}^{\infty}:=\mathcal{V}_{i}\times\{0\}\subset B\times H^{1},i\in
	\{1,\ldots,\ell\}$, since the projection $\mathbb{P}\left(  \mathcal{V}%
	_{i}\times\{0\}\right)  $ is contained in $B\times\mathbb{P}H^{1,0}$.
	
	The following theorem describes the Selgrade decomposition of the lifted flow
	$\Psi^{1}$. There is a unique Selgrade bundle for $\Psi^{1}$ which is not at
	infinity. We will call it the central Selgrade bundle and denote it by
	$\mathcal{V}_{c}^{1}$ (cf. also its spectral properties in Theorem
	\ref{Theorem_spectrum}).
	
	\begin{theorem}
		\label{Theorem_affine1}Consider an affine flow $\Psi$ on a vector bundle
		$B\times H$.
		
		(i) The Selgrade decomposition of the lifted flow $\Psi^{1}$ defined in
		Proposition \ref{Proposition_lift} is given by
		\begin{equation}
			B\times H^{1}=\mathcal{V}_{1}^{\infty}\oplus\cdots\oplus\mathcal{V}_{\ell^{+}%
			}^{\infty}\oplus\mathcal{V}_{c}^{1}\oplus\mathcal{V}_{\ell^{+}+\ell^{0}%
				+1}^{\infty}\oplus\cdots\oplus\mathcal{V}_{\ell}^{\infty}, \label{Sel1}%
		\end{equation}
		for some numbers $\ell^{+},\ell^{0}\geq0$ with $\ell^{+}+\ell^{0}\leq\ell$,
		and the central Selgrade bundle $\mathcal{V}_{c}^{1}$ is the unique Selgrade
		bundle having nonvoid intersection with $B\times H^{1,1}$.
		
		(ii) The intersection of the central Selgrade subbundle $\mathcal{V}_{c}^{1}$
		with the subbundle $B\times H^{1,0}$ is%
		\[
		\mathcal{V}_{c}^{1}\cap\left(  B\times H^{1,0}\right)  =\bigoplus_{i=\ell
			^{+}+1}^{i=\ell^{+}+\ell^{0}}\mathcal{V}_{i}^{\infty}=:\mathcal{V}_{c}%
		^{\infty}.
		\]

		(iii) The dimension of $\mathcal{V}_{c}^{1}$ is given by $\dim\mathcal{V}%
		_{c}^{1}=1+\dim\mathcal{V}_{c}^{\infty}$, and $\dim\mathcal{V}_{c}^{1}=1$
		holds if and only if $\mathcal{V}_{c}^{1}\cap\left(  B\times H^{1,0}\right)
		=B\times\{(0_{H},0)\}$.
		
		(iv) If $h^{1}(\mathcal{V}_{i})$ is chain transitive on the projective
		Poincar\'{e} bundle $B\times\mathbb{P}H^{1}$, then $\mathcal{V}_{i}^{\infty
		}\subset\mathcal{V}_{c}^{1}$.
	\end{theorem}
	
	\begin{proof}
		Theorem \ref{Theorem_Selgrade} applied to the linear flow $\Psi^{1}$ yields
		the Selgrade decomposition (\ref{Sel_1}) of $B\times H^{1}$. By Lemma
		\ref{Lemma_affine2} there is a unique Selgrade bundle $\mathcal{V}_{j}^{1}$
		with $\mathcal{V}_{j}^{1}\cap\left(  B\times H^{1,1}\right)  \not =%
		\varnothing$ and the other Selgrade bundles have the form $\mathcal{V}%
		_{i}\times\{0\}$. We write $\mathcal{V}_{c}^{1}:=\mathcal{V}_{j}^{1}$. Let
		$\ell^{0}$ the number of subbundles $\mathcal{V}_{i}\times\{0\}$ contained in
		$\mathcal{V}_{c}^{1}$. Since the chain recurrent components for the Selgrade
		bundles are linearly ordered, we can define $\ell^{+}\geq0$ such that the
		Selgrade decomposition has the form (\ref{Sel1}).
		
		The definitions imply that $\bigoplus_{i=\ell^{+}+1}^{i=\ell^{+}+\ell^{0}%
		}\mathcal{V}_{i}^{\infty}\subset\mathcal{V}_{c}^{1}$. Thus the assertion in
		(ii) follows from (\ref{dim0}), which in the present notation yields
		\[
		\dim\mathcal{V}_{c}^{1}=1+\sum_{i=\ell^{+}+1}^{i=\ell^{+}+\ell^{0}}%
		\dim\mathcal{V}_{i}^{\infty}.
		\]
		This also implies assertion (iii). In order to prove assertion (iv), suppose
		that $h^{1}(\mathcal{V}_{i})=\mathbb{P}(\mathcal{V}_{i}\times\{1\})$ is chain
		transitive. It follows that $\mathbb{P}(\mathcal{V}_{i}\times\{1\}$ is
		contained in the chain recurrent component $\mathcal{M}_{c}^{1}$, since the
		other chain recurrent components are $\mathbb{P}\mathcal{V}_{i}^{\infty}$,
		which are subsets of $B\times\mathbb{P}H^{1,0}$. For $(b,x)\in\mathcal{V}_{i}$
		and $n\in\mathbb{N}$ the sequence
		\[
		\mathbb{P}(b,x,\frac{1}{n})=\mathbb{P}(b,nx,1)\in\mathbb{P}\left(
		\mathcal{V}_{i}\times\{1\}\right)  \subset\mathcal{M}_{c}^{1}%
		\]
		converges for $n\rightarrow\infty$ to $\mathbb{P}(b,x,0)\in\mathbb{P}\left(
		\mathcal{V}_{i}\times\{0\}\right)  $, hence $\mathcal{V}_{i}^{\infty}%
		\subset\mathcal{V}_{c}^{1}$.
	\end{proof}
	
	\begin{remark}
		If there is an equilibrium $e\in B$ of $\theta$, i.e., $\theta_{t}%
		e=e,t\in\mathbb{R}$, with $f(e)=0\in L^{\infty}(\mathbb{R},H)$, it follows
		that the north pole $(e,0_{H},1)$ of the Poincar\'{e} sphere $\{e\}\times
		\mathbb{S}H^{1}$ is in $\mathcal{V}_{c}^{1}$. This holds since $(e,0_{H},1)$
		is an equilibrium of $\Psi^{1}$ implying $(e,\mathbb{P}\left(  0_{H},1\right)
		)\in\mathcal{M}_{c}^{1}$.
	\end{remark}
	
	Next we relate chain recurrence properties of the affine flow $\Psi$ on
	$B\times H$ and the flow $\mathbb{P}\Psi^{1}$ on the projective Poincar\'{e}
	bundle. Observe that the map $\left(  h^{1}\right)  ^{-1}$ may not preserve
	chain transitivity, since this is a homeomorphism between the non-compact
	spaces $B\times\mathbb{P}H^{1,1}$ and $B\times H$.
	
	\begin{corollary}
		\label{Corollary3.8}Consider an affine flow $\Psi$ on $B\times H$ with central
		Selgrade bundle $\mathcal{V}_{c}^{1}$ in $B\times H^{1}$.
		
		(i) If $(b,x)\in B\times H$ is chain recurrent for $\Psi$, then $h^{1}
		(b,x)\in\mathcal{M}_{c}^{1}$.
		
		(ii) The inclusion $\mathcal{M}_{c}^{1}\subset B\times\mathbb{P}H^{1,1}$ holds
		if and only if
		\[
		\mathcal{N}_{c}:=(h^{1})^{-1}\left(  \mathcal{M}_{c}^{1}\cap\left(
		B\times\mathbb{P}H^{1,1}\right)  \right)
		\]
		is compact. In this case $\mathcal{N}_{c}=\left(  h^{1}\right)  ^{-1}%
		(\mathcal{M}_{c}^{1})\ $is the chain recurrent set of $\Psi$.
	\end{corollary}
	
	\begin{proof}
		(i) By Proposition \ref{Proposition_e}(iii) any chain recurrent point $(b,x)$
		of $\Psi$ is mapped to a chain recurrent point $h^{1}(b,x)$ of $\mathbb{P}%
		\Psi^{1}$. Since $\mathcal{M}_{c}^{1}$ is the only chain recurrent component
		of $\mathbb{P}\Psi^{1}$ intersecting $B\times\mathbb{P}H^{1,1}$, it follows
		that $h^{1}(b,x)\in\mathcal{M}_{c}^{1}$.
		
		(ii) Let $\mathcal{M}_{c}^{1}\subset B\times\mathbb{P}H^{1,1}$. Since the flow
		$\mathbb{P}\Psi^{1}$ restricted to the compact connected chain recurrent set
		$\mathcal{M}_{c}^{1}$ is chain transitive, it follows that also $\mathcal{N}%
		_{c}$ is compact, connected, and chain transitive, and by (i) $\mathcal{N}%
		_{c}$ is the chain recurrent set. Conversely, if $\mathcal{N}_{c}$ is compact,
		also $h^{1}(\mathcal{N}_{c})=\mathcal{M}_{c}^{1}\cap\left(  B\times
		\mathbb{P}H^{1,1}\right)  $ is compact. Define neighborhoods of $h^{1}%
		(\mathcal{N}_{c})$ and $B\times\mathbb{P}H^{1,0}$ in $B\times\mathbb{P}H^{1}$
		by
		\[%
		\begin{array}
			[c]{l}%
			N_{1}(\varepsilon)=\left\{  \mathbb{P}(b,x,1)\left\vert \exists\mathbb{P}%
			(b^{\prime},x^{\prime},1)\in h^{1}(\mathcal{N}_{c}):d(\mathbb{P}%
			(b,x,1),\mathbb{P}(b^{\prime},x^{\prime},1))<\varepsilon\right.  \right\}  ,\\
			N_{2}(\varepsilon)=\left\{  \mathbb{P}(b,x,1)\left\vert \exists\mathbb{P}%
			(b^{\prime},x^{\prime},0)\in B\times\mathbb{P}H^{1,0}:d(\mathbb{P}%
			(b,x,1),\mathbb{P}(b^{\prime},x^{\prime},0))<\varepsilon\right.  \right\}  ,
		\end{array}
		\]
		respectively. The sets $B\times\mathbb{P}H^{1,0}$ and $h^{1}(\mathcal{N}_{c})$
		are disjoint compact sets, hence there is $\varepsilon>0$ such that
		$N_{1}(\varepsilon)\cap N_{2}(\varepsilon)=\varnothing$. Since the connected
		set $\mathcal{M}_{c}^{1}$ is contained in the union of the disjoint open sets
		$N_{1}(\varepsilon)$ and $N_{2}(\varepsilon)$, it follows that $\mathcal{M}%
		_{c}^{1}\cap N_{2}(\varepsilon)=\varnothing$, hence $\mathcal{M}_{c}%
		^{1}\subset B\times\mathbb{P}H^{1,1}$.
	\end{proof}
	
	\begin{remark}
		Although $\mathcal{N}_{c}$ is always nonvoid, the trivial example $\dot{x}=1$
		shows that $\Psi$ may have no chain recurrent point. Note that $\mathcal{M}%
		_{c}^{1}\subset B\times\mathbb{P}H^{1,1}$ is equivalent to $\mathcal{V}%
		_{c}^{1}\cap(B\times H^{1,0})=B\times\{\left(  0_{H},0\right)  \}$.
	\end{remark}
	
	Next we discuss the Morse spectrum of the Selgrade bundles; cf. Subsection
	\ref{Subsection2.4}.
	
	\begin{theorem}
		\label{Theorem_spectrum}(i) For an affine flow $\Psi$ with linear part $\Phi$
		the Morse spectrum of the central Selgrade bundle $\mathcal{V}_{c}^{1}$
		satisfies $\Sigma_{Mo}(\mathcal{V}_{i};\Phi)\subset\Sigma_{Mo}(\mathcal{V}
		_{c}^{1};\Psi^{1})$ for every $i\in\left\{  \ell^{+}+1,\ldots,\ell^{+}
		+\ell^{0}\right\}  $.
		
		(ii) If the flow $\Psi$ has a periodic trajectory, then $\mathcal{V}_{c}^{1}$
		is the unique Selgrade bundle containing the lift $\Psi_{t}^{1}(b,x,1),t\in
		\mathbb{R}$, of any periodic trajectory of $\Psi$, and
		\[
		\mathrm{co}\left\{  \{0\}\cup\bigcup\nolimits_{i=\ell^{+}+1}^{\ell^{+}%
			+\ell^{0}}\Sigma_{Mo}(\mathcal{V}_{i};\Phi)\right\}  \subset\Sigma
		_{Mo}(\mathcal{V}_{c}^{1};\Psi^{1}).
		\]

		(iii) For all $i\in\left\{  1,\ldots,\ell\right\}  $ the Morse spectra of the
		Selgrade bundles at infinity satisfy%
		\[
		\Sigma_{Mo}(\mathcal{V}_{i}^{\infty};\Psi^{1})=\Sigma_{Mo}(\mathcal{V}%
		_{i};\Phi).
		\]
		
	\end{theorem}
	
	\begin{proof}
		(i) According to Theorem \ref{Theorem_affine1} $\mathcal{V}_{i}\times
		\{0\}\subset\mathcal{V}_{c}^{1}$ for all $i\in\{\ell^{+}+1,\ldots,\ell
		^{+}+\ell^{0}\}$. Thus for all $\varepsilon,T>0$ any $(\varepsilon,T)$-chain
		$\zeta$ with $(b_{0},\mathbb{P}x_{0}),\ldots,(b_{n},\mathbb{P}x_{n})$ for
		$\mathbb{P}\Phi$ in $\mathbb{P}\mathcal{V}_{i}$ yields an $(\varepsilon
		,T)$-chain $\zeta^{1}$ for $\mathbb{P}\Psi^{1}$ in $\mathbb{P}\left(
		\mathcal{V}_{i}\times\{0\}\right)  \subset\mathbb{P}\mathcal{V}_{c}^{1}$ with
		$(b_{0},\mathbb{P}(x_{0},0)),\ldots,\allowbreak(b_{n},\mathbb{P}(x_{n},0))$.
		This follows since, by the definition of the distance in $\mathbb{P}H^{1}$ and
		$\mathbb{P}H$ in (\ref{metric_P}),
		\begin{align*}
			d(\mathbb{P}(x,0),\mathbb{P}(y,0))  &  =\min\left\{  \left\Vert \frac
			{(x,0)}{\left\Vert (x,0)\right\Vert }-\frac{(y,0)}{\left\Vert (y,0)\right\Vert
			}\right\Vert ,\left\Vert \frac{(x,0)}{\left\Vert (x,0)\right\Vert }%
			+\frac{(y,0)}{\left\Vert (y,0)\right\Vert }\right\Vert \right\} \\
			&  =\min\left\{  \left\Vert \frac{x}{\left\Vert x\right\Vert }-\frac
			{y}{\left\Vert y\right\Vert }\right\Vert ,\left\Vert \frac{x}{\left\Vert
				x\right\Vert }+\frac{y}{\left\Vert y\right\Vert }\right\Vert \right\}
			=d(\mathbb{P}x,\mathbb{P}y).
		\end{align*}
		The definition of $\Psi^{1}$ shows that $\Psi^{1}(t,b,x,0)=(\Phi(t,b,x),0)$
		for all $(t,b,x)\in\mathbb{R}\times B\times H$. Hence, with total time
		$\tau=\sum_{i=0}^{n-1}T_{i}$, the exponential growth rates of $\zeta^{1}$ and
		$\zeta$ are
		\begin{align*}
			\lambda(\zeta^{1})  &  =\frac{1}{\tau}\sum_{i=0}^{n-1}\left(  \log\left\Vert
			\Psi^{1}(T_{i},b_{i},x_{i},0)\right\Vert -\log\left\Vert (b_{i},x_{i}%
			,0)\right\Vert \right) \\
			&  =\frac{1}{\tau}\sum_{i=0}^{n-1}\left(  \log\left\Vert (\Phi(T_{i}%
			,b_{i},x_{i}),0)\right\Vert -\log\left\Vert (b_{i},x_{i},0)\right\Vert
			\right)  =\lambda(\zeta).
		\end{align*}
		This implies $\Sigma_{Mo}(\mathcal{V}_{i};\Phi)\subset\Sigma_{Mo}%
		(\mathcal{V}_{c}^{1};\Psi^{1})$ for $\mathcal{V}_{i}\times\{0\}\subset
		\mathcal{V}_{c}^{1}$.
		
		(ii) Suppose that the flow $\Psi$ has a periodic solution satisfying
		$\Psi_{\tau}(b,x)\allowbreak=(b,x)$ for some $\tau>0$. This yields a periodic
		solution of $\Psi^{1}$ given by $\Psi_{\tau}^{1}(b,x,1)\allowbreak=(b,x,1)\in
		B\times H^{1,1}$ implying that $\mathbb{P}(b,x,1)\in B\times\mathbb{P}H^{1,1}$
		is in a chain recurrent component of $\mathbb{P}\Psi^{1}$ and by Theorem
		\ref{Theorem_affine1} $\mathbb{P}(b,x,1)\in\mathcal{M}_{c}^{1}$. Thus the
		central Selgrade bundle of $\Psi^{1}$ is the Selgrade bundle containing the
		lift of any periodic trajectory of $\Psi$. The $\tau$-periodic trajectory of
		$\Psi^{1}$ yields $(\varepsilon,T)$-chains $\zeta^{k}$ (without jumps) with
		exponential growth rates $\lambda(\zeta^{k})=0$: Define for any $k\in
		\mathbb{N}$ the chain $\zeta^{k}$ with $n=1$ by%
		\[
		T_{0}=k\tau,(b_{0},x_{0},1)=(b_{1},x_{1},1)=(b,x,1).
		\]
		Then $\left\Vert \Psi^{1}(T_{0},b_{0},x_{0},1)\right\Vert =\left\Vert \Psi
		^{1}(k\tau,b,x,1)\right\Vert =\left\Vert (b,x,1)\right\Vert $ and
		$\lambda(\zeta^{k})=0$. The assertion on the convex hull follows, since by
		Theorem \ref{Theorem_Morse} the Morse spectrum of a Selgrade bundle is an interval.
		
		(ii) By Proposition \ref{Proposition_e}(i) the flows $\mathbb{P}\Phi$ on
		$B\times\mathbb{P}H$ and $\mathbb{P}\Psi^{1}$ restricted to $B\times
		\mathbb{P}H^{1,0}$ are conjugate. Thus the $(\varepsilon,T)$-chains in
		$B\times\mathbb{P}H$ correspond to $(\varepsilon^{\prime},T)$-chains in
		$B\times\mathbb{P}H^{1,0}$ with $\varepsilon\rightarrow0$ if and only if
		$\varepsilon^{\prime}\rightarrow0,$ and also the exponential growth rates of
		the corresponding chains coincide.
	\end{proof}
	
	\section{Split affine flows\label{Section4}}
	
	In this section we determine the central Selgrade bundle for a class of affine
	flows, which can be split into a linear, homogeneous part and an inhomogeneous part.
	
	We consider the following class of affine flows. The base space of the vector
	bundle is the product $B_{1}\times B_{2}$ of compact metric spaces $B_{1}$ and
	$B_{2}$. We suppose that chain transitive flows $\theta^{1}$ on $B_{1}$ and
	$\theta^{2}$ on $B_{2}$ are given. It follows from Theorem
	\ref{Theorem_chains} that this is equivalent to chain transitivity of the
	product flow $\theta_{t}(b^{1},b^{2})=(\theta_{t}^{1}b^{1},\theta_{t}^{2}
	b^{2}),t\in\mathbb{R}$, on $B_{1}\times B_{2}$. Furthermore, we suppose that
	there is an equilibrium of $\theta^{1}$ denoted by $e^{1}\in B_{1}$, hence
	$\theta_{t}^{1}e^{1}=e^{1},t\in\mathbb{R}$.
	
	\begin{definition}
		A split affine flow is an affine flow $\Psi$ on a vector bundle $\left(
		B_{1}\times B_{2}\right)  \times H$ of the form
		\[
		\Psi_{t}(b^{1},b^{2},x)=\left(  \theta_{t}^{1}b^{1},\Phi_{t}(b^{2},x)+\int
		_{0}^{t}\left(  \theta_{t}^{1}b^{1},\Phi_{t-s}(\theta_{s}^{2}b^{2}%
		,f(b^{1},s))\right)  ds\right)  ,
		\]
		where $\Phi$ is a linear flow on $B_{2}\times H$ and $f:B_{1}\rightarrow
		L^{\infty}(\mathbb{R},H),f(b^{1},s):=f(b^{1})(s),s\in\mathbb{R}$, satisfies
		\[
		f(e^{1})=0\mbox{ and }f(b^{1},t+s)=f(\theta_{s}^{1}b^{1},t)\mbox{ for all
		}b^{1}\in B_{1}\mbox{ and almost all }t,s\in\mathbb{R}.
		\]
		
	\end{definition}
	
	Note that the base flow on $B_{1}\times B_{2}$ of $\Psi$ is $\theta$, and
	\[
	\Psi_{t}(e^{1},b^{2},x)=(e^{1},\Phi_{t}(b^{2},x)),t\in\mathbb{R}%
	,\mbox{ for all }b^{2}\in B_{2},x\in H.
	\]
	In a trivial way, every linear flow may be viewed as a split affine flow:
	Define $B_{1}:=\{e^{1}\}$ and $f(e^{1})=0\in L^{\infty}(\mathbb{R}%
	,\mathbb{R})$. Linear control systems and, more generally, split affine
	control systems define split affine control flows; cf. Section \ref{Section6}.
	
	\begin{lemma}
		\label{Lemma_A}The linear part of $\Psi$ is the flow $\widetilde{\Phi}%
		_{t}(b^{1},b^{2},x)=\left(  \theta_{t}^{1}b^{1},\Phi_{t}(b^{2},x)\right)
		,\allowbreak t\in\mathbb{R}$, on $B_{1}\times B_{2}\times H$, and the Selgrade
		bundles of $\widetilde{\Phi}$ are given by $B_{1}\times\mathcal{V}_{i}$, where
		$\mathcal{V}_{i}\subset B_{2}\times H,i\in\{1,\ldots,\ell\}$, are the Selgrade
		bundles of $\Phi$.
	\end{lemma}
	
	\begin{proof}
		By the definitions, $\widetilde{\Phi}$ is the linear part of $\Psi$. By
		Theorem \ref{Theorem_Selgrade} the Selgrade decomposition is the finest
		decomposition into exponentially separated subbundles. Hence the Selgrade
		bundles $\mathcal{V}_{i}$ are exponentially separated. Since the two
		components $\theta_{t}^{1}b^{1}$ and $\Phi_{t}(b^{2},x)$ are independent, it
		follows that the subbundles $B_{1}\times\mathcal{V}_{i}$ are exponentially
		separated. Theorem \ref{Theorem_chains} implies that the product flow on
		$B_{1}\times\mathbb{P}\mathcal{V}_{i}$ is chain transitive, hence the
		subbundles $B_{1}\times\mathcal{V}_{i}$ are the Selgrade bundles.
	\end{proof}
	
	Any subbundle $\mathcal{V}\subset B_{2}\times H$ which is invariant for $\Phi$
	yields the invariant subbundle $B_{1}\times\mathcal{V}$ for $\widetilde{\Phi}%
	$. For $b^{2}\in B_{2}$ the points $(b^{2},0_{H},\pm1)\in B_{2}\times
	H\times\mathbb{R}=B_{2}\times H^{1}$ are the poles of the Poincar\'{e} sphere
	$\{b^{2}\}\times\mathbb{S}^{d}$. Define the polar subbundle $\mathcal{P}$ of
	$B_{2}\times H^{1}$ by
	\begin{equation}
		\mathcal{P}:=B_{2}\times\{0_{H}\}\times\mathbb{R}=\left\{  (b^{2},0_{H},r)\in
		B_{2}\times H\times\mathbb{R}\left\vert b^{2}\in B_{2},r\in\mathbb{R}\right.
		\right\}  . \label{P}%
	\end{equation}
	Then $\dim\mathcal{P}=1$ and $\mathcal{P}$ is a line bundle containing all
	poles. It is invariant for $\Phi_{t}^{1}(b^{2},x,r):=(\Phi_{t}(b^{2}%
	,x),r),t\in\mathbb{R}$. The set $\{e^{1}\}\times\mathcal{P}$ is invariant for
	the lift $\Psi^{1}$ to $B_{1}\times B_{2}\times H^{1}$. For a Selgrade bundle
	$\mathcal{V}_{i}$ of $\Phi$ the subbundle $\mathcal{V}_{i}^{\infty
	}=\mathcal{V}_{i}\times\{0\}$ of $B_{2}\times H^{1}$ yields the invariant
	subbundle $B_{1}\times\mathcal{V}_{i}^{\infty}$ of $B_{1}\times B_{2}\times
	H^{1}$ for $\Psi^{1}$. By Lemma \ref{Lemma_A} the subbundles $B_{1}%
	\times\mathcal{V}_{i}^{\infty}$ are the subbundles at infinity for $\Psi^{1}$.
	
	\begin{theorem}
		\label{Theorem_split}For a split affine flow $\Psi$ on $B_{1}\times
		B_{2}\times H$ with lift $\Psi^{1}$ to $B_{1}\times B_{2}\times H^{1}$ the
		central Selgrade bundle $\mathcal{V}_{c}^{1}\subset B_{1}\times B_{2}\times
		H^{1}$ satisfies
		\[
		\mathcal{V}_{c}^{1}\cap\left(  \{e^{1}\}\times B_{2}\times H^{1}\right)
		=\{e^{1}\}\times\left(  \mathcal{P}\oplus\bigoplus\nolimits_{i}\mathcal{V}
		_{i}^{\infty}\right)  ,
		\]
		where $\mathcal{P}$ is the polar bundle and the sum is taken over all indices
		$i\in\{1,\ldots,\ell\}$ such that $h^{1}(\mathcal{V}_{i})=\mathbb{P}
		(\mathcal{V}_{i}\times\{1\})\subset B_{2}\times\mathbb{P}H^{1}$ is chain transitive.
	\end{theorem}
	
	\begin{proof}
		Theorem \ref{Theorem_affine1} yields that the central Selgrade bundle
		$\mathcal{V}_{c}^{1}$ is the unique Selgrade bundle of $\Psi^{1}$ such that
		$\mathcal{M}_{c}^{1}\cap\left(  B_{1}\times B_{2}\times\mathbb{P}%
		H^{1,1}\right)  \not =\varnothing$. The set $\mathbb{P}\mathcal{P}=B_{2}%
		\times\mathbb{P}\left(  \{0\}\times\{1\}\right)  \subset B_{2}\times
		\mathbb{P}H^{1,1}$ is chain transitive, since $B_{2}$ is chain transitive. It
		follows that also $\{e^{1}\}\times\mathbb{P}\mathcal{P}$ is chain transitive.
		Thus the set $\{e^{1}\}\times\mathbb{P}\mathcal{P}\subset B_{1}\times\left(
		B_{2}\times\mathbb{P}H^{1,1}\right)  $ is contained in a chain transitive
		component of $\mathbb{P}\Psi^{1}$, hence in $\mathcal{M}_{c}^{1}$. This
		implies that $\{e^{1}\}\times\mathcal{P}\subset\mathcal{V}_{c}^{1}$. We claim
		that
		\begin{equation}
			\mathcal{V}_{c}^{1}\cap\left(  \{e^{1}\}\times B_{2}\times H^{1}\right)
			=\{e^{1}\}\times\left(  \mathcal{P}\oplus\bigoplus\nolimits_{i}\mathcal{V}%
			_{i}^{\infty}\right)  , \label{4.3}%
		\end{equation}
		where $I$ is the set of all indices $i\in\{1,\ldots,\ell\}$ such that
		$B_{1}\times\mathcal{V}_{i}^{\infty}\subset\mathcal{V}_{c}^{1}$. In fact, the
		inclusion \textquotedblleft$\supset$\textquotedblright\ is clear. Since
		$B_{1}\times\mathcal{V}_{i}^{\infty}$ are the subbundles at infinity for
		$\Psi^{1}$, Theorem \ref{Theorem_affine1}(iii) shows that the dimension of
		$\mathcal{V}_{c}^{1}$ is
		\[
		1+\sum\nolimits_{i\in I}\dim\left(  B_{1}\times\mathcal{V}_{i}^{\infty
		}\right)  =1+\sum\nolimits_{i\in I}\dim\mathcal{V}_{i}^{\infty}.
		\]
		This equals the dimension of $\mathcal{P}\oplus\bigoplus\nolimits_{i}%
		\mathcal{V}_{i}^{\infty}$, hence equality (\ref{4.3}) holds.
		
		It remains to show that the summation in (\ref{4.3}) can be taken over all $i$
		such that $\mathbb{P}(\mathcal{V}_{i}\times\{1\})$ is chain transitive. If
		$h^{1}(\mathcal{V}_{i})=\mathbb{P}\left(  \mathcal{V}_{i}\times\{1\}\right)  $
		is chain transitive, then $\{e^{1}\}\times\mathbb{P}\left(  \mathcal{V}%
		_{i}\times\{1\}\right)  $ is chain transitive, and as in the proof of Theorem
		\ref{Theorem_affine1}(iv) it follows that $\{e^{1}\}\times\mathcal{V}%
		_{i}^{\infty}\subset\mathcal{V}_{c}^{1}$. Conversely, suppose that
		$\{e^{1}\}\times\mathcal{V}_{i}^{\infty}\subset\mathcal{V}_{c}^{1}$. Equality
		(\ref{4.3}) implies that for $(e^{1},(b^{2},x))\in B_{1}\times\mathcal{V}%
		_{i}$
		\[
		(e^{1},b^{2},x,1)=(e^{1},b^{2},0_{H},1)\oplus(e^{1},b^{2},x,0)\in\left(
		\{e^{1}\}\times\mathcal{P}\right)  \oplus\left(  \{e^{1}\}\times
		\mathcal{V}_{i}^{\infty}\right)  \subset\mathcal{V}_{c}^{1}\mbox{.}
		\]
		This shows that $\{e^{1}\}\times\mathcal{V}_{i}\times\{1\}\subset
		\mathcal{V}_{c}^{1}$, hence $\{e^{1}\}\times\mathbb{P}\left(  \mathcal{V}%
		_{i}\times\{1\}\right)  \subset\mathcal{M}_{c}^{1}$ is chain transitive. It
		follows that $\mathbb{P}\left(  \mathcal{V}_{i}\times\{1\}\right)  $ is chain transitive.
	\end{proof}
	
	\begin{remark}
		\label{Remark_linear2}Theorem \ref{Theorem_split} applies, in particular, to
		linear flows $\Phi$, where $B_{1}$ is trivial and hence may be omitted. The
		lift $\Phi^{1}$ has the form $\Phi_{t}^{1}(b,x,r)=\left(  \Phi_{t}%
		(b,x),r\right)  $ for $(b,x,r)\in B\times H\times\mathbb{R}$, and the points
		$(b,0_{H},\pm1)$ are the poles of the Poincar\'{e} sphere $\{b\}\times
		\mathbb{S}^{d}$. The central Selgrade bundle satisfies
		\[
		\mathcal{V}_{c}^{1}=\mathcal{P}\oplus\bigoplus\nolimits_{i}\mathcal{V}%
		_{i}^{\infty},
		\]
		where $\mathcal{P}=B\times\{0_{H}\}\times\mathbb{R}$ is the polar bundle and the
		sum is taken over all indices $i\in\{1,\ldots,\ell\}$ such that $h^{1}%
		(\mathcal{V}_{i})\subset B\times\mathbb{P}H^{1}$ is chain transitive.
	\end{remark}
	
	We have seen that the subbundles $\mathcal{V}_{i}$ for linear flows $\Phi$,
	which yield chain transitive sets on the projective Poincar\'{e} bundle, play
	a special role. The paper Colonius \cite{Col23} has discussed the lift of
	linear flows to $B\times H^{1}$ and chain transitivity for the projection to
	the northern hemisphere of the Poincar\'{e} sphere bundle. The following
	theorem formulates similar results in the projective Poincar\'{e} bundle.
	Since the proofs are completely analogous, we omit them.
	
	\begin{theorem}
		\label{Corollary4.4}Let $\mathcal{V}_{i}$ be a Selgrade bundle of a linear
		flow $\Phi$ on $B\times H$. Then the following assertions are equivalent:
		
		(a) The set $h^{1}(\mathcal{V}_{i})=\mathbb{P}(\mathcal{V}_{i}\times\{1\})$ is
		chain transitive in the projective Poincar\'{e} bundle $B\times\mathbb{P}
		H^{1}$.
		
		(b) The subbundle $\mathcal{V}_{i}$ contains a line $l=\{(b,\alpha
		x_{0})\left\vert \alpha\in\mathbb{R}\right.  \}$ for some $x_{0}\not =0$ such
		that $h^{1}(l)$ is chain transitive in $B\times\mathbb{P}H^{1}$.
		
		A sufficient condition for (b) (or (a)) is $0\in\mathrm{int}\Sigma
		_{Mo}(\mathcal{V}_{i};\Phi)$.
	\end{theorem}
	
	\begin{proof}
		It is clear that (a) implies (b). The converse follows using the same
		construction as in the proof of \cite[Theorem 4.3]{Col23}. The last assertion
		follows as \cite[Theorem 4.7]{Col23}.
	\end{proof}
	
	\begin{remark}
		Recall that by Theorem \ref{Theorem_spectrum} the Morse spectrum $\Sigma
		_{Mo}(\mathcal{V}_{c}^{1};\Psi^{1})$ of the central Selgrade bundle
		$\mathcal{V}_{c}^{1}$ contains $0$ if $\Psi$ possesses a periodic trajectory.
		If $0\in\mathrm{int}\Sigma_{Mo}(\mathcal{V}_{c}^{1};\Psi^{1})$ one can apply
		Theorem \ref{Corollary4.4} to the linear flow $\Psi^{1}$. With $H^{2}
		:=H^{1}\times\mathbb{R}$ and
		\[
		h^{2}:B\times H^{1}\rightarrow B\times\mathbb{P}H^{2},\,h^{2}
		(b,x,r):=(b,\mathbb{P}(x,r,1)),
		\]
		one deduces that $h^{2}(\mathcal{V}_{c}^{1})=\mathbb{P}\left(  \mathcal{V}
		_{c}^{1}\times\{1\}\right)  $ is chain transitive on $B\times\mathbb{P}H^{2}$.
	\end{remark}
	
	\section{Uniformly hyperbolic affine flows\label{Section5}}
	
	In this section we determine for uniformly hyperbolic affine flows the central
	Selgrade bundle for the lifted flow $\Psi^{1}$.
	
	First we define uniformly hyperbolic affine flows; cf. Colonius and Santana
	\cite{ColSan11}.
	
	\begin{definition}
		An affine flow $\Psi$ on $B\times H$ with linear part $\Phi$ is uniformly
		hyperbolic if $\Phi$ admits a decomposition $B\times H=\mathcal{V}^{1}
		\oplus\mathcal{V}^{2}$ into $\Phi$-invariant subbundles $\mathcal{V}^{1}$ and
		$\mathcal{V}^{2}$ such that
		
		(i) the restrictions $\Phi_{t}^{i}(b,x)=\left(  \theta_{t}b,\varphi
		^{i}(t,b,x)\right)  $ to $\mathcal{V}^{i},i=1,2$, satisfy for constants
		$\alpha,K>0$ and for all $b\in B$
		\[
		\left\Vert \Phi_{t}^{1}(b,\cdot)\right\Vert \leq Ke^{-\alpha t}\mbox{ for
		}t\geq0\mbox{ and }\left\Vert \Phi_{t}^{2}(b,\cdot)\right\Vert \leq
		Ke^{\alpha t}\mbox{ for }t\leq0,
		\]

		(ii) there is $M>0$ with $\left\Vert f(b)\right\Vert _{\infty}\leq M$ for all
		$b\in B$, and the following maps defined on $B$ with values in $H$ are
		continuous:
		\[
		b\mapsto\int_{-\infty}^{0}\varphi^{1}(-s,\theta_{s}b,f(b,s))ds\mbox{ and
		}b\mapsto\int_{-\infty}^{0}\varphi^{2}(s,\theta_{-s}b,f(b,-s))ds.
		\]
		
	\end{definition}
	
	The next result follows by \cite[Corollary 1 and Theorem 2.5]{ColSan11}.
	
	\begin{theorem}
		\label{Theorem5.1}Consider a uniformly hyperbolic affine flow $\Psi$ on
		$B\times H$ with linear part $\Phi$.
		
		(i) Then for every $b\in B$ there is a unique bounded solution $(\theta
		_{t}b,e(b,t)),\allowbreak t\in\mathbb{R},$ for the flow $\Psi$ and the map
		$e:\mathbb{R}\times B\rightarrow H$ is continuous.
		
		(ii) The affine flow $\Psi$ and its homogeneous part $\Phi$ are conjugate by
		the homeomorphism
		\[
		h_{aff}=(id_{B},h_{aff}^{0}):B\times H\rightarrow B\times H\mbox{ satisfying
		}h_{aff}(\Psi_{t}(b,x))=\Phi_{t}(h_{aff}(b,x)),
		\]
		where $h_{aff}^{0}(b,x)=x-e(b,0),b\in B$.
	\end{theorem}
	
	Note that
	\[
	h_{aff}^{0}(\theta_{t}b,\psi(t,b,x))=\varphi(t,b,h_{aff}^{0}(b,x))\mbox{ for
		all }t\in\mathbb{R},b\in B,x\in H.
	\]

	Again we assume throughout that the base space $B$ is chain transitive. The
	following result characterizes the chain recurrent set for hyperbolic affine flows.
	
	\begin{theorem}
		\label{Theorem_hyperbolic1}Suppose that $\Psi$ is a uniformly hyperbolic
		affine flow. Then the chain recurrent set of the linear part $\Phi$ of $\Psi$
		is $\mathcal{R}=B\times\{0_{H}\}$ and $h_{aff}(\mathcal{R}%
		)=\{(b,-e(b,0))\left\vert b\in B\right.  \}$ is the chain recurrent set for
		the affine flow $\Psi$. The set $h_{aff}(\mathcal{R})$ is compact and chain transitive.
	\end{theorem}
	
	\begin{proof}
		For the linear flow $\Phi$ every chain recurrent point in the stable subbundle
		$\mathcal{V}^{1}$ is contained in the product $B\times\{0_{H}\}$, which is
		chain transitive, and the same holds for the unstable bundle $\mathcal{V}^{2}%
		$. Since $B\times H=\mathcal{V}^{1}\oplus\mathcal{V}^{2}$ it follows that the
		chain recurrent set of $\Phi$ is $\mathcal{R}=B\times\{0_{H}\}$. For the proof
		of these assertions, note that similar arguments as for Antunez, Mantovani,
		and Var\~{a}o \cite[Corollary 2.11]{Antunez} can be used, where hyperbolic
		linear operators on Banach spaces are considered. By Theorem \ref{Theorem5.1}
		\[
		h_{aff}(\mathcal{R})=\{(b,h_{aff}^{0}(b,x))\left\vert (b,x)\in\mathcal{R}%
		\right.  \}=\{(b,-e(b,0))\left\vert b\in B\right.  \}.
		\]
		Thus $h_{aff}(\mathcal{R})$ is compact since $B$ is compact and $e(\cdot,0)$
		is continuous.
		
		The map $h_{aff}$ is uniformly continuous: In fact, for $\varepsilon>0$ it
		follows by compactness of $B$ and continuity of $e(\cdot,0)$ that there is
		$\delta(\varepsilon)\in(0,\varepsilon/2)$ such that $d(b,b^{\prime}%
		)<\delta(\varepsilon)$ and $\left\Vert x-x^{\prime}\right\Vert <\delta
		(\varepsilon)$ implies
		\[
		\left\Vert x-e(b,0)-\left(  x^{\prime}-e(b^{\prime},0)\right)  \right\Vert
		\leq\left\Vert x-x^{\prime}\right\Vert +\left\Vert e(b,0)-e(b^{\prime
		},0)\right\Vert <\delta(\varepsilon)+\varepsilon/2<\varepsilon.
		\]
		Hence $d(b,x),(b^{\prime},x^{\prime}))<\delta(\varepsilon)$ implies
		$d(h_{aff}(b,x),h_{aff}(b^{\prime},x^{\prime}))<\varepsilon$. Analogously one
		proves that the inverse of $h_{aff}$ given by
		\[
		h_{aff}^{-1}(b,x)=(b,x+e(b,0))
		\]
		is uniformly continuous. Let $\varepsilon,T>0$ and consider $h(b,0_{H}%
		),h(b^{\prime},0_{H})\in h_{aff}(\mathcal{R})$ with $b,b^{\prime}\in B$. By
		chain transitivity of $B$ there is a $(\delta(\varepsilon),T)$-chain in
		$B\times\{0_{H}\}$ from $(b,0_{H})$ to $(b^{\prime},0_{H})$. Then $h_{aff}$
		maps it onto an $(\varepsilon,T)$-chain from $h_{aff}(b,0_{H})$ to
		$h_{aff}(b^{\prime},0_{H})$. Since $\varepsilon,T>0$ are arbitrary, this
		proves that $h_{aff}(\mathcal{R})$ is chain transitive.
		
		It remains to prove that $h_{aff}(\mathcal{R})$ is the chain recurrent set of
		$\Psi$. Let $\varepsilon>0$. By uniform continuity of $h_{aff}^{-1}$ there is
		$\delta^{\prime}(\varepsilon)>0$ such that $d(b,x),(b^{\prime},x^{\prime
		}))<\delta^{\prime}(\varepsilon)$ implies $d(h_{aff}^{-1}(b,x),h_{aff}%
		^{-1}(b^{\prime},x^{\prime}))<\varepsilon$. For any chain recurrent point
		$(b,x)$ of $\Psi$ and $T>0$ there is a $(\delta^{\prime}(\varepsilon
		),T)$-chain from $(b,x)$ to $(b,x)$. This is mapped by $h_{aff}^{-1}$ to an
		$(\varepsilon,T)$-chain of $\Phi$ from $h_{aff}^{-1}(b,x)$ to $h_{aff}%
		^{-1}(b,x)$. This proves that $h_{aff}^{-1}(b,x)\in\mathcal{R}$ and hence
		$(b,x)=h_{aff}(h_{aff}^{-1}(b,x))\in h_{aff}(\mathcal{R})$.
	\end{proof}
	
	Next we determine the Selgrade bundles and their Morse spectra.
	
	\begin{theorem}
		\label{Theorem_hyperbolic2}Suppose that $\Psi$ is a uniformly hyperbolic
		affine flow.
		
		(i) Then the Selgrade bundles of $\Psi^{1}$ are $\mathcal{V}_{i}^{\infty}%
		,i\in\{1,\ldots,\ell\}$, together with the central Selgrade bundle
		$\mathcal{V}_{c}^{1}$, which is the line bundle in $B\times H^{1}$ given by
		\begin{equation}
			\mathcal{V}_{c}^{1}=\{(b,-re(b,0),r)\in B\times H\times\mathbb{R}\left\vert
			b\in B,r\in\mathbb{R}\right.  \}. \label{Vstar2}%
		\end{equation}
		The projection $\mathcal{M}_{c}^{1}=\mathbb{P}\mathcal{V}_{c}^{1}$ to
		$B\times\mathbb{P}H^{1}$ is a compact subset of $B\times\mathbb{P}H^{1,1}$ and
		coincides with the image of the chain recurrent set of $\Psi$, i.e.,%
		\begin{equation}
			\mathcal{M}_{c}^{1}=\left\{  (b,\mathbb{P}(x,1))\left\vert (b,x)\in
			h_{aff}(\mathcal{R})\right.  \right\}  . \label{E_hyp}%
		\end{equation}

		(ii) The Morse spectra of the Selgrade bundles are%
		\[
		\Sigma_{Mo}(\mathcal{V}_{i}^{\infty};\Psi^{1})=\Sigma_{Mo}(\mathcal{V}%
		_{i};\Phi)\text{ for }i\in\left\{  1,\ldots,\ell\right\}  \text{ and }%
		\Sigma_{Mo}(\mathcal{V}_{c}^{1};\Psi^{1})=\{0\}.
		\]
		
	\end{theorem}
	
	\begin{proof}
		(i) By Theorem \ref{Theorem_hyperbolic1} the chain recurrent set of the affine
		flow $\Psi$ is $h_{aff}(\mathcal{R})=\{(b,-e(b,0))\left\vert b\in B\right.
		\}$, and it is compact and chain transitive. Denote by $\mathcal{V}_{\ast}%
		^{1}$ the right hand side of (\ref{Vstar2}). First we claim that
		$\mathcal{V}_{\ast}^{1}$ is a subbundle. The projection of $\mathcal{V}_{\ast
		}^{1}$ to $B\times\mathbb{P}H^{1}$ satisfies
		\begin{align*}
			\mathbb{P}\mathcal{V}_{\ast}^{1}  &  =\{(b,\mathbb{P}\left(  -e(b,0),1\right)
			)\in B\times H\times\mathbb{R}\left\vert b\in B\right.  \}\\
			&  =\{(b,\mathbb{P}\left(  x,1\right)  )\in B\times H\times\mathbb{R}%
			\left\vert (b,x)\in h_{aff}(\mathcal{R})\right.  \}=h^{1}(h_{aff}%
			(\mathcal{R})).
		\end{align*}
		By Proposition \ref{Proposition_e}(ii) the compact and chain transitive set
		$h_{aff}(\mathcal{R})$ is mapped to the compact set $\mathbb{P}\mathcal{V}%
		_{\ast}^{1}=h^{1}(h_{aff}(\mathcal{R}))\subset B\times H^{1,1}$ which is chain
		transitive for $\mathbb{P}\Psi^{1}$.
		
		For every $b\in B$ the fiber $\left\{  (b,-re(b,0),r),r\in\mathbb{R}\right\}
		$, is one dimensional and $\mathcal{V}_{\ast}^{1}$ is closed. In fact, suppose
		that a sequence $(b_{n},-r_{n}e(b_{n},0),r_{n}),n\in\mathbb{N\,}$ in
		$\mathcal{V}_{\ast}^{1}$ converges to $(b,x,r)\in B\times H\times\mathbb{R}$.
		Then $b_{n}\rightarrow b$ and $r_{n}\rightarrow r$, and by continuity of
		$e(\cdot,0)$ it follows that $r_{n}e(b_{n},0)\rightarrow re(b,0)$. This shows
		that $(b,x,r)=(b,-re(b,0),r)\in\mathcal{V}_{\ast}^{1}$. According to Colonius
		and Kliemann \cite[Lemma B.1.13]{ColK00} it follows that $\mathcal{V}_{\ast
		}^{1}$ is a one dimensional subbundle of $B\times H^{1}$.
		
		By Proposition \ref{Proposition_e}(i) the sets $\mathbb{P}\mathcal{V}%
		_{i}^{\infty}\subset B\times\mathbb{P}H^{1,0}$ are chain recurrent components
		of $\mathbb{P}\Psi^{1}$ restricted to $B\times\mathbb{P}H^{1,0}$, hence they
		are chain transitive for $\mathbb{P}\Psi^{1}$. Furthermore, the intersection
		satisfies%
		\[
		\mathcal{V}_{\ast}^{1}\cap\bigoplus\nolimits_{i=1}^{\ell}\mathcal{V}%
		_{i}^{\infty}=B\times\{\left(  0_{H},0\right)  \}\subset B\times H^{1},
		\]
		since $r=0$ implies $re(b,0)=0$. It follows that
		\begin{equation}
			\mathcal{V}_{\ast}^{1}\oplus\bigoplus\nolimits_{i=1}^{\ell}\mathcal{V}%
			_{i}^{\infty}=B\times H^{1}, \label{dec_hyp1}%
		\end{equation}
		since the fibers on the left hand side have dimension $d+1$. The sets
		$\mathbb{P}\mathcal{V}_{\ast}^{1}$ and $\mathbb{P}\mathcal{V}_{i}^{\infty}$
		are contained in chain recurrent components $\mathcal{M}^{1}$ and
		$\mathcal{M}_{j}^{1}$ with $j$ in some index set $J$, respectively, of
		$\mathbb{P}\Psi^{1}$. Lemma \ref{Lemma_dim1}(ii) implies that, actually, the
		sets $\mathbb{P}\mathcal{V}_{\ast}^{1}$ and $\mathbb{P}\mathcal{V}_{i}%
		^{\infty}$ are chain recurrent components, since otherwise the subbundles for
		$\mathcal{M}^{1}$ and $\mathcal{M}_{j}^{1}$ would satisfy
		\[
		\dim\left(  \mathcal{V}^{1}\oplus\bigoplus\nolimits_{j\in J}\mathcal{V}%
		_{j}^{1}\right)  >d+1=\dim(B\times H^{1}),
		\]
		which is a contradiction. It follows that $\mathcal{V}_{\ast}^{1}$ and
		$\mathcal{V}_{i}^{\infty}$ are Selgrade bundles, and $J=\{1,\ldots,\ell\}$.
		Thus (\ref{dec_hyp1}) is a decomposition into Selgrade bundles, and Theorem
		\ref{Theorem_affine1}(i) shows that $\mathcal{V}_{\ast}^{1}=\mathcal{V}%
		_{c}^{1}$. Since $\dim\mathcal{V}_{c}^{1}=1$, Theorem \ref{Theorem_affine1}%
		(iii) implies that $\mathcal{M}_{c}^{1}\subset B\times\mathbb{P}H^{1,1}$.
		Equality (\ref{E_hyp}) is a consequence of Theorem \ref{Theorem_hyperbolic1}.
		
		(ii) The assertion for the Selgrade bundles $\mathcal{V}_{i}^{\infty}$ follows
		by Theorem \ref{Theorem_spectrum}(iii). For the central Selgrade bundle
		$\mathcal{V}_{c}^{1}$ equality (\ref{Vstar2}) implies that the projection to
		the projective bundle is%
		\[
		\mathbb{P}\mathcal{V}_{c}^{1}=\{(b,\mathbb{P}(-e(b,0),1)\left\vert b\in
		B\right.  \}.
		\]
		Consider an $(\varepsilon,T)$ chain in $\mathbb{P}\mathcal{V}_{c}^{1}$ given
		by $T_{0},\ldots,T_{n-1}\geq T$, and $(b_{0},p_{0}),\ldots,\allowbreak
		(b_{n},p_{n})\in B\times\mathbb{P}\mathcal{V}_{c}^{1}$ with $d(\mathbb{P}%
		\Phi(T_{i},b_{i},p_{i}),(b_{i+1},p_{i+1}))<\varepsilon$ for $i=0,\ldots,n-1.$
		Then $p_{i}=\mathbb{P}(e(b_{i},0),1)$ and with total time $\tau=\sum
		_{i=0}^{n-1}T_{i}$ the exponential growth rate of $\zeta$ is
		\[
		\lambda(\zeta)=\frac{1}{\tau}\sum_{i=0}^{n-1}\left(  \log\left\Vert \Psi
		^{1}(T_{i},-e(b_{i},0),1)\right\Vert -\log\left\Vert (e(b_{i},0),1)\right\Vert
		\right)  .
		\]
		By definition and Theorem \ref{Theorem5.1}(i)%
		\begin{equation}
			\Psi^{1}(T_{i},-e(b_{i},0),1)=(\theta_{T_{i}}b_{i},-e(b_{i},T_{i}),1).
			\label{bdd}%
		\end{equation}
		Recall that by assumption $\left\Vert f(b)\right\Vert _{\infty}\leq M$ for all
		$b\in B$. This implies that the bounded solutions $e(b,t),t\in\mathbb{R}$, are
		uniformly bounded for $b\in B$ (cf. Colonius and Santana \cite[formula (13)
		and Corollary 1]{ColSan11}. Thus by (\ref{bdd}) also $\left\Vert \Psi
		^{1}(T_{i},e(b_{i},0),1)\right\Vert $ is uniformly bounded. It follows that
		for $T$ large enough and $T_{i}>T$%
		\begin{align*}
			\lambda(\zeta)  &  =\sum_{i=0}^{n-1}\frac{T_{i}}{\sum_{j=0}^{n}T_{j}}\frac
			{1}{T_{i}}\left(  \log\left\Vert \Psi^{1}(T_{i},-e(b_{i},0),1)\right\Vert
			-\log\left\Vert (-e(b_{i},0),1)\right\Vert \right) \\
			&  \leq\sum_{i=0}^{n-1}\frac{T_{i}}{\sum_{j=0}^{n-1}T_{j}}\varepsilon
			=\varepsilon.
		\end{align*}
		Since $\varepsilon>0$ is arbitrary, it follows that $\Sigma_{Mo}%
		(\mathcal{V}_{c}^{1};\Psi^{1})=\{0\}$.
	\end{proof}
	
	\begin{remark}
		For a linear uniformly hyperbolic flow $\Phi$ the bounded solutions are given
		by $(\theta_{t}b,0_{H}),t\in\mathbb{R}$, hence the central Selgrade bundle of
		the lift $\Phi^{1}$ coincides with the polar bundle $\mathcal{P}$ (cf.
		(\ref{P}))
		\[
		\mathcal{V}_{c}^{1}=\{(b,0_{H},r)\in B\times H^{1}\left\vert b\in
		B,r\in\mathbb{R}\right.  \}=\mathcal{P}.
		\]
		
	\end{remark}
	
	\section{Control systems and examples\label{Section6}}
	
	In this section we study control systems which provide a rich class of affine
	flows. After introducing some notation for control systems, the existence and
	uniqueness of chain control sets in $\mathbb{R}^{d}$ is analyzed. Then we
	apply the results of the previous sections to affine control flows defined by
	affine control systems with bounded control range and present several examples.
	
	\subsection{Control systems}
	
	Control-affine systems have the form
	\begin{equation}
		\dot{x}(t)=X_{0}(x(t))+\sum_{i=1}^{m}u_{i}(t)X_{i}(x(t)),\,u(t)=(u_{1}%
		(t),\ldots,u_{m}(t))\in\Omega,\label{control1}%
	\end{equation}
	where $X_{0},X_{1},\ldots,X_{m}$ are smooth ($C^{\infty}$-)vector fields on a
	manifold $M$ and $\Omega\subset\mathbb{R}^{m}$. We assume that for every
	admissible control $u$ in
	\[
	\mathcal{U}:=\{u\in L^{\infty}(\mathbb{R},\mathbb{R}^{m})\left\vert
	u(t)\in\Omega\mbox{
		for almost all }t\right.  \}
	\]
	and every initial state $x(0)=x_{0}\in M$ there exists a unique
	(Carath\'{e}odory) solution $\psi(t,x_{0},u),t\in\mathbb{R}$.
	
	Suppose that the control range $\Omega\subset\mathbb{R}^{m}$ is a convex and
	compact neighborhood of $0\in\mathbb{R}^{m}$, endow the set $\mathcal{U}$ of
	controls with a metric compatible with the weak$^{\ast}$ topology on
	$L^{\infty}(\mathbb{R},\mathbb{R}^{m})$, and fix a metric (compatible with the
	topology) on $M$. The control flow is defined as $\Psi:\mathbb{R}%
	\times\mathcal{U}\times M\rightarrow\mathcal{U}\times M,\,(t,u,x_{0}%
	)\mapsto(u(t+\cdot),\psi(t,x_{0},u))$, where $u(t+\cdot)(s):=u(t+s),s\in
	\mathbb{R}$, is the right shift. The control flow $\Psi$ is continuous and
	$\mathcal{U}$ is compact and chain transitive; cf. Colonius and Kliemann
	\cite[Chapter 4]{ColK00} or Kawan \cite[Section 1.4]{Kawa13}.
	
	Maximal chain transitive sets of a control flow enjoy a characterization in
	the state space $M$ of the control system. Fix $x,y\in M$ and let
	$\varepsilon,T>0.$ A controlled $(\varepsilon,T)$\textit{-}chain $\zeta$ from
	$x$ to $y$ is given by $n\in\mathbb{N},\ x_{0}=x,\ldots,x_{n-1},x_{n}=y\in
	M,\ u_{0},\ldots,u_{n-1}\in\mathcal{U}$, and $T_{0},\ldots,T_{n-1}\geq T$
	with
	\[
	d(\psi(T_{j},x_{j},u_{j}),x_{j+1})<\varepsilon\mbox{ }\,\mbox{for\thinspace
		all}\,\,\,j=0,\ldots,n-1.
	\]
	Define a chain control set of system (\ref{control1}) as a maximal nonvoid set
	$E\subset M$ such that (i) for all $x\in E$ there is $u\in\mathcal{U}$ such
	that $\psi(t,x,u)\in E$ for all $t\in\mathbb{R}$ and (ii) for all $x,y\in E$
	and $\varepsilon,\,T>0$ there is a controlled $(\varepsilon,T)$-chain from $x$
	to $y$.
	
	For control affine systems of the form above, \cite[Proposition 1.24]{Kawa13}
	shows that a chain control set $E$ yields a maximal invariant chain transitive
	set $\mathcal{E}$ of the control flow $\Psi$ via
	\begin{equation}
		\mathcal{E}:=\{(u,x)\in\mathcal{U}\times M\left\vert \psi(t,x,u)\in E\mbox{
			for all }t\in\mathbb{R}\right.  \}, \label{chain_transitive1}%
	\end{equation}
	and for any maximal invariant chain transitive set in $\mathcal{U}\times M$
	the projection to $M$ is a chain control set.
	
	\subsection{Affine control systems}
	
	General affine control system have the form
	\begin{equation}
		\dot{x}(t)=A_{0}x(t)+a_{0}+\sum\nolimits_{i=1}^{m}u_{i}(t)[A_{i}%
		x(t)+a_{i}],\,u\in\mathcal{U}, \label{control_affine}%
	\end{equation}
	where $A_{i}\in\mathbb{R}^{d\times d},a_{i}\in\mathbb{R}^{d},i\in
	\{0,1,\ldots,m\}$. If the control range $\Omega\subset\mathbb{R}^{m}$ is a
	convex and compact neighborhood of $0\in\mathbb{R}^{m}$, the system generates
	an affine control flow $\Psi$ on $\mathcal{U}\times\mathbb{R}^{d}$. We also
	consider the following special case.
	
	\begin{definition}
		Split affine control systems have the form
		\begin{equation}
			\dot{x}(t)=\left[  A_{0}+\sum\nolimits_{i=1}^{p}v_{i}(t)A_{i}\right]
			x(t)+Bu(t)=A(v(t))x(t)+Bu(t), \label{qaffine1}%
		\end{equation}
		where $A_{0},A_{1},\ldots,A_{p}\in\mathbb{R}^{d\times d},A(v):=A_{0}%
		+\sum\nolimits_{i=1}^{p}v_{i}A_{i}$ for $v\in\mathbb{R}^{p}$, and
		$B\in\mathbb{R}^{d\times m}$. The set of admissible controls is
		\[
		\mathcal{U}_{1}\times\mathcal{U}_{2}=\{(u,v)\in L^{\infty}(\mathbb{R}%
		,\mathbb{R}^{m})\times L^{\infty}(\mathbb{R},\mathbb{R}^{p})\left\vert
		u(t)\in\Omega_{1},v(t)\in\Omega_{2}\text{ for }t\in\mathbb{R}\right.  \},
		\]
		where $\Omega_{1}\subset\mathbb{R}^{m}$ and $\Omega_{2}\subset\mathbb{R}^{p}$.
	\end{definition}
	
	Split affine control systems are affine control systems: Define $A_{i}%
	^{\prime}:=0$ for $i=1,\ldots,m$, and $u_{m+i}:=v_{i}$ and $A_{m+i}^{\prime
	}=A_{i}$ for $i=1,\ldots,p$. Furthermore, denote the columns of $B$ by
	$a_{i}^{\prime},i=1,\ldots,m$, and let $a_{i}^{\prime}:=0,i=m+1,\ldots
	,\allowbreak m+p$. Then, with $A_{0}^{\prime}:=A_{0}$ and $a_{0}^{\prime}:=0$,
	system equation (\ref{qaffine1}) is equivalent to
	\[
	\dot{x}(t)=A_{0}^{\prime}x(t)+a_{0}^{\prime}+\sum_{i=1}^{m+p}u_{i}(t)\left[
	A_{i}^{\prime}x(t)+a_{i}^{\prime}\right]
	\]
	with controls in $\mathcal{U}:=\{u\in L^{\infty}(\mathbb{R},\mathbb{R}%
	^{m+p})\left\vert u(t)\in\Omega:=\Omega_{1}\times\Omega_{2}\text{ for }%
	t\in\mathbb{R}\right.  \}$.
	
	The following theorem presents results on existence and uniqueness of chain
	control sets for split affine control systems in $\mathbb{R}^{d}$. The
	considered systems may not generate a control flow, since the assumptions on
	the control range are more general. Thus a chain control set need not be
	related to a chain transitive component of a flow.
	
	\begin{theorem}
		\label{Theorem_unique}For every split affine control system of the form
		(\ref{qaffine1}), where $0\in\Omega_{2}$ and the control range $\Omega_{1}$ is
		a convex neighborhood of $0\in\mathbb{R}^{m}$, there exists a unique chain
		control set $E$ in $\mathbb{R}^{d}$.
	\end{theorem}
	
	\begin{proof}
		First note that for $u\equiv0$ the origin $0\in\mathbb{R}^{d}$ is an
		equilibrium, hence there exists a chain control set $E$ with $0\in E$. The
		trajectories $x(t)=\psi(t,x_{0},u,v),t\in\mathbb{R}$, of (\ref{qaffine1})
		satisfy for $\alpha\in(0,1)$
		\[
		\alpha\dot{x}(t)=A_{0}\alpha x(t)+\sum\nolimits_{i=1}^{m}v_{i}(t)\alpha
		x(t)+B\alpha u(t).
		\]
		It follows that
		\begin{equation}
			\psi(t,\alpha x_{0},\alpha u,v)=\alpha\psi(t,x_{0},u,v),t\in\mathbb{R},
			\label{alpha}%
		\end{equation}
		and $\psi(\cdot,\alpha x_{0},\alpha u,v)$ is a trajectory of (\ref{qaffine1}),
		since $\Omega_{1}$ is a convex neighborhood of $0\in\mathbb{R}^{m}$ implying
		that the controls $\alpha u$ are in $\mathcal{U}_{1}$.
		
		Suppose that $E^{\prime}$ is any chain control set and let $x\in E^{\prime}$.
		First we will construct controlled $(\varepsilon,T)$-chains from $x$ to $0\in
		E$.
		
		\textbf{Step 1:} There is a controlled $(\varepsilon,T)$-chain from $x$ to
		$\alpha x$ for some $\alpha\in(0,1)$.
		
		For the proof consider a controlled $(\varepsilon/2,T)$-chain $\zeta$ in
		$E^{\prime}$ from $x$ to $x$ given by $x_{0}=x,x_{1},\ldots,x_{n}
		=x,\,(u_{0},v_{0})\ldots,(u_{n-1},v_{n-1})\in\mathcal{U}_{1}\times
		\mathcal{U}_{2}$, and $T_{0},\ldots,T_{n-1}>T$ with
		\[
		\left\Vert \psi(T_{i},x_{i},u_{i},v_{i})-x_{i+1}\right\Vert <\varepsilon
		/2\mbox{ for }i=0,\ldots,n-1.
		\]
		Let $\alpha\in(0,1)$ with $(1-\alpha)\left\Vert x\right\Vert <\varepsilon/2$,
		hence
		\[%
		\begin{array}
			[c]{l}%
			\left\Vert \psi(T_{n-1},x_{n-1},u_{n-1},v_{n-1})-\alpha x_{n}\right\Vert \\
			\leq\left\Vert \psi(T_{n-1},x_{n-1},u_{n-1},v_{n-1})-x\right\Vert +\left\Vert
			x-\alpha x\right\Vert <\varepsilon.
		\end{array}
		\]
		This defines a controlled $(\varepsilon,T)$-chain $\zeta^{(1)}$ from $x$ to
		$\alpha x$.
		
		\textbf{Step 2: }Replacing $x_{i}$ by $\alpha x_{i}$ and $u_{i}$ by $\alpha
		u_{i}$ for all $i$ we get by (\ref{alpha})
		\[
		\left\Vert \psi(T_{i},\alpha x_{i},\alpha u_{i},v_{i})-\alpha x_{i+1}
		\right\Vert =\alpha\left\Vert \psi(T_{i},x_{i},u_{i},v_{i})-x_{i+1}\right\Vert
		<\varepsilon/2
		\]
		and
		\[%
		\begin{array}
			[c]{l}%
			\left\Vert \psi(T_{n-1},\alpha x_{n-1},\alpha u_{n-1},v_{n-1})-\alpha
			^{2}x\right\Vert \\
			\leq\left\Vert \psi(T_{n-1},\alpha x_{n-1},\alpha u_{n-1},v_{n-1})-\alpha
			x\right\Vert +\left\Vert \alpha x-\alpha^{2}x\right\Vert <\varepsilon.
		\end{array}
		\]
		This defines a controlled $(\varepsilon,T)$-chain $\zeta^{(2)}$ from $\alpha
		x$ to $\alpha^{2}x$. The concatenation of $\zeta^{(2)}$ and $\zeta^{(1)}$
		yields a controlled $(\varepsilon,T)$-chain $\zeta^{(2)}\circ\zeta^{(1)}$ from
		$x$ to $\alpha^{2}x$.
		
		Repeating this construction, we find that the concatenation $\zeta^{(k)}%
		\circ\cdots\circ\zeta^{(1)}$ is a controlled $(\varepsilon,T)$-chain from
		$x\in E^{\prime}$ to $\alpha^{k}x$. Since $\alpha^{k}\rightarrow0$ for
		$k\rightarrow\infty$, we can take $k\in\mathbb{N}$ large enough, such that the last piece of the chain $\zeta^{(k)}$ satisfies%
		\[
		\left\Vert \psi(T_{n-1},\alpha^{k-1}x_{n-1},\alpha^{k-1}u_{n-1},v_{n-1}%
		)\right\Vert <\varepsilon.
		\]
		Thus we may take $0\in E$ as the final point of this controlled chain showing
		that the concatenation $\zeta^{(k)}\circ\cdots\circ\zeta^{(1)}$ define a
		controlled $(\varepsilon,T)$-chain from $x\in E^{\prime}$ to $0\in E$.
		
		\textbf{Step 3:} Together with (\ref{qaffine1}) we consider the time reversed
		system
		\begin{equation}
			\dot{y}(t)=-\left[  A_{0}+\sum\nolimits_{i=1}^{p}v_{i}(t)A_{i}\right]
			y(t)-Bu(t),\quad(u,v)\in\mathcal{U}_{1}\times\mathcal{U}_{2},
			\label{time_reversed}%
		\end{equation}
		with trajectories $\psi^{-}(t,y,u,v),t\in\mathbb{R}$. For $S>0$ and
		$z:=\psi^{-}(S,y,u,v)$ the trajectories are related by
		\begin{equation}
			\psi^{-}(t,y,u,v)=\psi(S-t,z,u(S-\cdot),v(S-\cdot))\mbox{ for }t\in
			\lbrack0,S]. \label{rev1}%
		\end{equation}
		This holds, since the right hand side of (\ref{rev1}) satisfies
		\[%
		\begin{array}
			[c]{l}%
			\frac{d}{dt}\psi(S-t,z,u(S-\cdot),v(S-\cdot))=-\dot{\psi}(S-t,z,u(S-\cdot
			),v(S-\cdot))\\
			=-\left[  A_{0}+\sum\nolimits_{i=1}^{p}v_{i}(S-t)A_{i}\right]  \psi
			(S-t,z,u(S-\cdot),v(S-\cdot))-Bu(S-t)
		\end{array}
		\]
		with $\psi(S-S,z,u(S-\cdot),v(S-\cdot))=z$.
		
		The chain control sets of the time reversed system coincide with the chain
		control sets of the original system. Using the relation
		(\ref{chain_transitive1}) of chain control sets and maximal chain transitive
		sets, this follows from the fact that chain transitive sets are invariant
		under time reversal (Colonius and Kliemann \cite[Proposition 3.1.13(ii)]%
		{ColK14}) or it can be proved directly (using similar arguments as below).
		
		The result from Step 2 can be applied to the time reversed system
		(\ref{time_reversed}) and yields controlled $(\varepsilon,T)$-chains from
		$x\in E^{\prime}$ to $0\in E$. Let $\zeta^{-}$ be such a chain, given by
		$y_{0}=x,y_{1},\ldots,y_{n-1},y_{n}=0,\,(u_{0},v_{0}),\ldots,(u_{n-1}%
		,v_{n-1})\in\mathcal{U}_{1}\times\mathcal{U}_{2}$, and $T_{0},\ldots
		,T_{n-1}>T$ with
		\[
		\left\Vert \psi^{-}(T_{i},y_{i},u_{i},v_{i})-y_{i+1}\right\Vert <\varepsilon
		\mbox{ for }i=0,\ldots,n-1.
		\]
		Define a controlled $(\varepsilon,T)$-chain $\zeta$ for (\ref{control_linear})
		by going backwards in $\zeta^{-}$: The point $0\in\mathbb{R}^{d}$ is an
		equilibrium for control $u=0,v=0$, hence define $x_{0}=0,u_{0}=0,v_{0}%
		=0,T_{0}>T$, and for $i=1,\ldots,n$
		\[%
		\begin{array}
			[c]{l}%
			T_{i}^{\ast}:=T_{n-i},\,x_{i}:=\psi^{-}(T_{n-i},y_{n-i},u_{n-i},v_{n-i}%
			),\label{rev2}\\
			u_{i}^{\ast}(t):=u_{n-i}(T_{n-i}-t),\,\,v_{i}^{\ast}(t):=v_{n-i}%
			(T_{n-i}-t),t\in\lbrack0,T_{n-i}],
		\end{array}
		\]
		and let $x_{n+1}:=x$. This defines a controlled $(\varepsilon,T)$-chain from
		$x_{0}=0\in E$ to $x_{n+1}=x\in E^{\prime}$, since for $i=1,\ldots,n-1$,
		\[
		\left\Vert \psi(T_{i}^{\ast},x_{i},u_{i}^{\ast},v_{i}^{\ast})-x_{i+1}%
		\right\Vert =\left\Vert y_{n-i}-\psi^{-}(T_{n-i-1},y_{n-i-1},u_{n-i-1}%
		,v_{n-i-1})\right\Vert <\varepsilon.
		\]
		Together with Step 2 it follows that the chain control sets $E$ and
		$E^{\prime}$ coincide.
	\end{proof}
	
	Next we illustrate the results from Section \ref{Section3} on the Selgrade
	decomposition by the simplest case of autonomous differential equations.
	
	\begin{example}
		\label{Example6.1}Consider the autonomous affine differential equation
		$\dot{x}(t)=Ax(t)+a$ with $A\in\mathbb{R}^{d\times d}$ and $a\in\mathbb{R}%
		^{d}$. Here subbundles are just subspaces. The Selgrade subspaces of the
		linear part $\dot{x}=Ax$ are the Lyapunov spaces $L(\lambda_{i})$, which are
		the sums of the generalized real eigenspaces for eigenvalues $\mu$ with real
		part $\lambda_{i}$. The lifted system in $\mathbb{R}^{d}\times\mathbb{R}$ is
		described by
		\begin{equation}
			\left(
			\begin{array}
				[c]{c}%
				\dot{x}(t)\\
				\dot{z}(t)
			\end{array}
			\right)  =\left(
			\begin{array}
				[c]{cc}%
				A & a\\
				0 & 0
			\end{array}
			\right)  \left(
			\begin{array}
				[c]{c}%
				x(t)\\
				z(t)
			\end{array}
			\right)  . \label{aut_lift}%
		\end{equation}
		For the lifted system the eigenvalues are given by the eigenvalues of $A$
		together with the additional eigenvalue $\mu=0$. With the Lyapunov spaces at
		infinity $L(\lambda_{i})^{\infty}:=L(\lambda_{i})\times\{0\}$ the Selgrade
		decomposition has the form
		\[
		\mathbb{R}^{d+1}=L(\lambda_{1})\oplus\cdots\oplus L(\lambda_{\ell^{+}})\oplus
		L_{c}^{1}\oplus L(\lambda_{\ell^{+}+\ell^{0}+1})\oplus\cdots\oplus
		L(\lambda_{\ell});
		\]
		here $\lambda_{i}<0$ for $i\in\{1,\ldots,\ell^{+}\}$ and $\lambda_{i}>0$ for
		$i\in\{\ell^{+}+\ell^{0}+1,\ldots,\ell\}$ \ The number $\ell^{0}=0$ if and
		only if $A$ is hyperbolic and $\ell^{0}=1$ otherwise. The subspace $L_{c}^{1}$
		is the Lyapunov space for the Lyapunov exponent $\lambda=0$. In particular, if
		$A$ is hyperbolic, the unique bounded solution is the equilibrium
		$x_{0}=-A^{-1}a$, and by Theorem \ref{Theorem_hyperbolic2} the central
		Selgrade subspace is
		\[
		L_{c}^{1}=\{(rx_{0},r)\in\mathbb{R}^{d}\times\mathbb{R}\left\vert
		r\in\mathbb{R}\right.  \}\mbox{ with }\dim L_{c}^{1}=1.
		\]
		
	\end{example}
	
	\begin{remark}
		An in-depth analysis of nonautonomous affine differential equations is given
		in the classical treatise by Massera and Sch\"{a}ffer \cite{MasSch}.
	\end{remark}
	
	An application of Theorem \ref{Theorem_affine1} and Theorem
	\ref{Theorem_hyperbolic2} to affine control system (\ref{control_affine}) and
	the associated affine control flow $\Psi$ yields the following results. The
	map $f:\mathcal{U}\rightarrow L^{\infty}(\mathbb{R},\mathbb{R}^{m})$ is given
	by $f(u)(t)=a_{0}+\sum\nolimits_{i=1}^{m}u_{i}(t)a_{i},t\in\mathbb{R}$, and
	the linear part $\Phi$ of $\Psi$ is the linear control flow associated with
	the bilinear control system
	\begin{equation}
		\dot{x}(t)=\left[  A_{0}+\sum\nolimits_{i=1}^{m}u_{i}(t)A_{i}\right]
		x(t),\,u\in\mathcal{U}. \label{bilinear}%
	\end{equation}

	\begin{corollary}
		\label{Corollary_affine1}Consider an affine control system of the form
		(\ref{control_affine}), where the control range $\Omega$ is a convex and
		compact neighborhood of $0\in\mathbb{R}^{m}$, and denote by $\Psi$ the
		associated affine control flow on $\mathcal{U}\times\mathbb{R}^{d}$. For
		$i\in\{1,\ldots,\ell\}$ let $\mathcal{V}_{i}\subset\mathcal{U}\times
		\mathbb{R}^{d}$ be the Selgrade bundles of the linear flow $\Phi$ associated
		with control system (\ref{bilinear}), and let $\mathcal{V}_{i}^{\infty
		}=\mathcal{V}_{i}\times\{0\}$.
		
		(i) The Selgrade decomposition of the lifted flow $\Psi^{1}$ has the form
		\begin{equation}
			\mathcal{U}\times\mathbb{R}^{d+1}=\mathcal{V}_{1}^{\infty}\oplus\cdots
			\oplus\mathcal{V}_{\ell^{+}}^{\infty}\oplus\mathcal{V}_{c}^{1}\oplus
			\mathcal{V}_{\ell^{+}+\ell^{0}+1}^{\infty}\oplus\cdots\oplus\mathcal{V}_{\ell
			}^{\infty}, \label{Sel4}%
		\end{equation}
		for some numbers $\ell^{+},\ell^{0}\geq0$ with $\ell^{+}+\ell^{0}\leq\ell$.
		
		(ii) The central Selgrade bundle $\mathcal{V}_{c}^{1}$ satisfies
		\[
		\mathcal{V}_{c}^{1}\cap\left(  \mathcal{U}\times\mathbb{R}^{d}\times
		\{0\}\right)  =\bigoplus_{i=\ell^{+}+1}^{i=\ell^{+}+\ell^{0}}\mathcal{V}%
		_{i}^{\infty}:=\mathcal{V}_{c}^{\infty}\mbox{ and }
		\]

		(iii) The dimension of $\mathcal{V}_{c}^{1}$ is given by $\dim\mathcal{V}
		_{c}^{1}=1+\dim\mathcal{V}_{c}^{\infty}$, and $\dim\mathcal{V}_{c}^{1}=1$
		holds if and only if $\mathcal{V}_{c}^{1}\cap\left(  \mathcal{U}
		\times\mathbb{R}^{d}\times\{0\}\right)  =\mathcal{U}\times\{0\}\times\{0\}$.
		
		(iv) If (\ref{bilinear}) is uniformly hyperbolic, the central Selgrade bundle
		is the line bundle
		\begin{equation}
			\mathcal{V}_{c}^{1}=\{(u,-re(u,0),r)\in\mathcal{U}\times\mathbb{R}^{d}
			\times\mathbb{R}\left\vert u\in\mathcal{U},r\in\mathbb{R}\right.  \},
			\label{6.7}%
		\end{equation}
		where $e(u,t),t\in\mathbb{R}$, is the unique bounded solution of
		(\ref{control_affine}) for $u\in\mathcal{U}$, and $\mathcal{M}_{c}^{1}
		\subset\mathcal{U}\times\mathbb{P}^{d,1}$.
	\end{corollary}
	
	We can give a more explicit description of the central Selgrade bundle
	$\mathcal{V}_{c}^{1}$ for split affine control systems of the form
	(\ref{qaffine1}). Here we suppose that $\Omega_{1}$ and $\Omega_{2}$ are
	convex and compact neighborhoods of the origin. Hence the associated control
	flow $\Psi_{t}(u,v,x),t\in\mathbb{R}$, on $\mathcal{U}_{1}\times
	\mathcal{U}_{2}\times\mathbb{R}^{d}$ is a well defined split affine flow with
	compact metric spaces $B_{1}:=\mathcal{U}_{1},B_{2}:=\mathcal{U}_{2}$ and
	equilibrium $e^{1}:=0_{\mathcal{U}_{1}}\in\mathcal{U}_{1}$, and
	\begin{equation}
		\Psi_{t}(u,v,x)=(u(t+\cdot),v(t+\cdot),\psi(t,x,u,v))\in\mathcal{U}_{1}%
		\times\mathcal{U}_{2}\times\mathbb{R}^{d}. \label{split_control}%
	\end{equation}
	The homogeneous part is given by the bilinear control system
	\begin{equation}
		\dot{x}(t)=A(v(t))x(t),\,v\in\mathcal{U}_{2}\mbox{,} \label{bilinear0}%
	\end{equation}
	which does not depend on $u\in\mathcal{U}_{1}$.
	
	The following corollary is an immediate consequence of Theorem
	\ref{Theorem_split}.
	
	\begin{corollary}
		\label{Corollary_split}Consider the split affine control flow $\Psi$ given by
		(\ref{split_control}) associated with a control system of the form
		(\ref{qaffine1}). Then the central Selgrade bundle $\mathcal{V}_{c}^{1}$ of
		the lift $\Psi^{1}$ to $\mathcal{U}_{1}\times\mathcal{U}_{2}\times
		\mathbb{R}^{d}\times\mathbb{R}$ satisfies
		\[
		\mathcal{V}_{c}^{1}\cap\left(  \{0_{\mathcal{U}_{1}}\}\times\mathcal{U}%
		_{2}\times\mathbb{R}^{d+1}\right)  =\{0_{\mathcal{U}_{1}}\}\times\left(
		\mathcal{P}\oplus\bigoplus\nolimits_{i}\mathcal{V}_{i}^{\infty}\right)  .
		\]
		Here $\mathcal{P}:=\mathcal{U}_{2}\times\left(  \{0\}\times\mathbb{R}\right)
		\subset\mathcal{U}_{2}\times\mathbb{R}^{d+1}$ is the polar bundle,
		$\mathcal{V}_{i}\subset\mathcal{U}_{2}\times\mathbb{R}^{d},i\in\{1,\ldots
		,\ell\}$, are the Selgrade bundles of the homogeneous part (\ref{bilinear0}),
		and the sum is taken over all indices $i$ such that $h^{1}(\mathcal{V}%
		_{i})=\mathbb{P}(\mathcal{V}_{i}\times\{1\})\subset\mathcal{U}_{2}%
		\times\mathbb{P}^{d}$ is chain transitive.
	\end{corollary}
	
	\begin{remark}
		\label{Remark_linear} A particular case of (\ref{qaffine1}) are linear control
		systems, which have the form
		\begin{equation}
			\dot{x}(t)=Ax(t)+Bu(t),u\in\mathcal{U}, \label{control_linear}%
		\end{equation}
		with $A\in\mathbb{R}^{d\times d}$ and $B\in\mathbb{R}^{d\times m}$. Here
		$\mathcal{U}_{2}$ is trivial and omitted. The homogeneous part has a very
		simple structure, since it is determined by the autonomous differential
		equation $\dot{x}=Ax$. The corresponding Selgrade bundles are $\mathcal{V}%
		_{i}=\mathcal{U}\times L(\lambda_{i})$ with the Lyapunov spaces $L(\lambda
		_{i})$ of $A$. The polar subspace is $\mathcal{P}=\{0\}\times\mathbb{R}%
		\subset\mathbb{R}^{d+1}$ and the central Selgrade bundle satisfies
		\[
		\mathcal{V}_{c}^{1}\cap\left(  \{0_{\mathcal{U}}\}\times\mathbb{R}%
		^{d+1}\right)  =\{0_{\mathcal{U}}\}\times\bigoplus\nolimits_{i}\left(
		L(\lambda_{i})\times\mathbb{R}\right)  ,
		\]
		where the sum is taken over all indices $i$ such that $\mathbb{P}%
		(L(\lambda_{i})\times\{1\})$ is chain transitive.
	\end{remark}
	
	Next we exploit the relation between chain recurrent components of control
	flows and chain control sets. System (\ref{control_affine}) can be embedded
	into a bilinear control system in $\mathbb{R}^{d+1}$ of the form (cf. Elliott
	\cite[Subsection 3.8.1]{Elliott})%
	
	\begin{equation}
		\left(
		\begin{array}
			[c]{c}%
			\dot{x}(t)\\
			\dot{z}(t)
		\end{array}
		\right)  =\left(
		\begin{array}
			[c]{cc}%
			A_{0} & a_{0}\\
			0 & 0
		\end{array}
		\right)  \left(
		\begin{array}
			[c]{c}%
			x(t)\\
			z(t)
		\end{array}
		\right)  +\sum_{i=1}^{m}u_{i}(t)\left(
		\begin{array}
			[c]{cc}%
			A_{i} & a_{i}\\
			0 & 0
		\end{array}
		\right)  \left(
		\begin{array}
			[c]{c}%
			x(t)\\
			z(t)
		\end{array}
		\right)  ,u\in\mathcal{U},\label{hom_d+1}%
	\end{equation}
	with trajectories denoted by $\psi^{1}(t,x_{0},z_{0},u),t\in\mathbb{R}$. This
	control system induces a control system on projective space $\mathbb{P}^{d}$
	(cf., e.g., Colonius and Kliemann \cite[Chapter 6]{ColK00}) with trajectories
	$\mathbb{P}\psi^{1}(t,\mathbb{P}(x_{0},z_{0}),u),t\in\mathbb{R}$, for
	$(x_{0},z_{0})\not =(0,0)$. The linear control flow generated by
	(\ref{hom_d+1}) is the lift $\Psi^{1}$ of the control flow $\Psi$ for
	(\ref{control_affine}) and the control flow of the induced control system on
	$\mathbb{P}^{d}$ is the projective flow $\mathbb{P}\Psi^{1}$.
	
	Projective space $\mathbb{P}^{d}$ can be written as the disjoint union
	$\mathbb{P}^{d}=\mathbb{P}^{d,1}\dot{\cup}\mathbb{P}^{d,0}$, where
	$\mathbb{P}^{d,1}:=\left\{  \mathbb{P}(x,1)\left\vert x\in\mathbb{R}%
	^{d}\right.  \right\}  $ and $\mathbb{P}^{d,0}:=\left\{  \mathbb{P}%
	(x,0)\left\vert 0\not =x\in\mathbb{R}^{d}\right.  \right\}  $. Note that
	$\mathbb{P}^{d,1}$ can be identified with the northern hemisphere
	$\mathbb{S}^{d,+}$ of the unit sphere $\mathbb{S}^{d}$ and $\mathbb{P}^{d,0}$
	corresponds to the equator of $\mathbb{S}^{d}$.
	
	The following theorem clarifies the relation between chain control sets $E$ in
	$\mathbb{R}^{d}$ and the chain control set $E_{c}^{1}$ in projective
	Poincar\'{e} space $\mathbb{P}^{d}$.
	
	\begin{theorem}
		\label{Theorem_E}Consider an affine control system of the form
		(\ref{control_affine}), where the control range $\Omega$ is a convex and
		compact neighborhood of $0\in\mathbb{R}^{m}$.
		
		(i) Then there is a unique chain control set $E_{c}^{1}$ of the induced
		control system on the projective Poincar\'{e} space $\mathbb{P}^{d}$ such that
		$E_{c}^{1}\cap\mathbb{P}^{d,1}\not =\varnothing$. It is given by $E_{c}%
		^{1}=\{\mathbb{P}(x,r)\in\mathbb{P}^{d}\left\vert \exists u\in\mathcal{U}%
		:(u,\mathbb{P}(x,r))\in\mathcal{M}_{c}^{1}\right.  \}$, where $\mathcal{M}%
		_{c}^{1}$ is the projection of the central Selgrade bundle $\mathcal{V}%
		_{c}^{1}$.
		
		(ii) If there is a chain control set $E$ in $\mathbb{R}^{d}$ of the affine
		control system (\ref{control_affine}), the image $\mathbb{P}\left(
		E\times\{1\}\right)  $ in the projective Poincar\'{e} space $\mathbb{P}^{d}$
		is contained in $E_{c}^{1}$.
		
		(iii) If (\ref{control_affine}) is uniformly hyperbolic, then there is a
		unique chain control set $E$ in $\mathbb{R}^{d}$. It is compact and the chain
		control set $E_{c}^{1}$ given by the image of $E$, i.e., $E_{c}^{1}$
		$=\left\{  \mathbb{P}\left(  x,1\right)  \left\vert x\in E\right.  \right\}
		$, is a compact subset of $\mathbb{P}^{d,1}$. For every $u\in\mathcal{U}$
		there exists a unique element $x\in E$ with $\psi(t,x,u)\in E$ for all
		$t\in\mathbb{R}$.
	\end{theorem}
	
	\begin{proof}
		(i) The correspondence (\ref{chain_transitive1}) between maximal invariant
		chain transitive sets of the control flow and chain control sets implies that
		there is a chain control set $E_{c}^{1}$ in $\mathbb{P}^{d}$ with
		\[
		\mathcal{M}_{c}^{1}=\{\left(  u,\mathbb{P}\left(  x,r\right)  \right)
		\left\vert \mathbb{P}\psi^{1}(t,\mathbb{P}\left(  x,r\right)  ,u)\in E_{c}%
		^{1}\mbox{ for }t\in\mathbb{R}\right.  \}.
		\]
		Since $\mathcal{M}_{c}^{1}$ is the only chain recurrent component of
		$\mathbb{P}\Psi^{1}$ having a nonvoid intersection with $\mathcal{U}%
		\times\mathbb{P}^{d,1}$, it follows that $E_{c}^{1}$ is the unique chain
		control set with $E_{c}^{1}\cap\mathbb{P}^{d,1}\not =\varnothing$.
		
		(ii) Let $E\subset\mathbb{R}^{d}$ be a chain control set of
		(\ref{control_affine}). An application of Corollary \ref{Corollary3.8}(i)
		shows that the maximal chain transitive set $\mathcal{E}$ of the affine
		control flow $\Psi$ associated with $E$ satisfies $h^{1}(\mathcal{E}%
		)\subset\mathcal{M}_{c}^{1}$. By (i) it follows that $\mathbb{P}\left(
		E\times\{1\}\right)  \subset E_{c}^{1}$.
		
		(iii) The assertions follow by Theorem \ref{Theorem_hyperbolic2}: The chain
		recurrent set $h_{aff}(\mathcal{R})$ of $\Psi$ is compact and chain transitive
		and is mapped onto the chain transitive set $\mathcal{M}_{c}^{1}$. Thus
		$h_{aff}(\mathcal{R})$ corresponds to the unique chain control set $E$ of
		control system (\ref{control_affine}), and $\mathcal{M}_{c}^{1}$ corresponds
		to the chain control set $E_{c}^{1}$ of the control system on $\mathbb{P}^{d}%
		$. Since $\mathcal{M}_{c}^{1}$ is a compact subset of $\mathcal{U}%
		\times\mathbb{P}^{d,1}$ it follows that $E_{c}^{1}$ is a compact subset of
		$\mathbb{P}^{d,1}$. The last assertion follows, since $\mathcal{V}_{c}^{1}$ is
		one dimensional.
	\end{proof}
	
	We briefly indicate how for linear control systems of the form
	(\ref{control_linear}) stronger results can be obtained under additional
	assumptions. Suppose that the matrices $A,B$ satisfy $\mathrm{rank}%
	[B,AB,\ldots,A^{d-1}B]=d$. Define a control set $D$ as a maximal nonvoid set
	in $\mathbb{R}^{d}$ such that (i) for all $x\in D$ there is a control
	$u\in\mathcal{U}$ with $\psi(t,x,u)\in D$ for all $t\geq0$ and (ii) for all
	$x,y\in D$ and all $\varepsilon>0$ there are $u\in\mathcal{U}$ and $T>0$ with
	$\left\Vert \psi(T,x,u)-y\right\Vert <\varepsilon$. Then one can deduce from
	Sontag \cite[Corollary 3.6.7]{Son98} that there is a unique control set $D$
	with nonvoid interior and
	\[
	L(0)\subset D\subset L(0)+F,
	\]
	where $F$ is a compact and convex subset of $\mathbb{R}^{d}$. The map
	$e_{S}:\mathbb{R}^{d}\rightarrow$ $\mathbb{S}^{d,+},x\mapsto\frac
	{(x,1)}{\left\Vert (x,1)\right\Vert }$ to the northern hemisphere of the
	Poincar\'{e} sphere is a homeomorphism. By Colonius, Santana, and Setti
	\cite[Theorem 15(ii)]{ColSS23b} the induced control system on $\mathbb{S}%
	^{d,+}$ has a unique control set with nonvoid interior, which is given by
	$e_{S}(D)$, and its intersection with the equator $\mathbb{S}^{d,0}$
	satisfies
	\begin{equation}
		\overline{e_{S}(D)}\cap\mathbb{S}^{d,0}=\overline{e_{S}(L(0))}\cap
		\mathbb{S}^{d,0}. \label{intersection}%
	\end{equation}
	For the projective Poincar\'{e} space $\mathbb{P}^{d}$ it similarly follows
	that $\mathbb{P}\left(  D\times\{1\}\right)  $ is a control set with nonvoid
	interior in $\mathbb{P}^{d,1}$ and its closure in $\mathbb{P}^{d}$ satisfies
	\[
	\overline{\mathbb{P}\left(  D\times\{1\}\right)  }\cap\mathbb{P}%
	^{d,0}=\overline{\mathbb{P}(L(0)\times\{1\})}\cap\mathbb{P}^{d,0}.
	\]
	Since $\mathbb{P}\left(  D\times\{1\}\right)  $ is a control set with nonvoid
	interior, Kawan \cite[Proposition 1.24(ii)]{Kawa13} implies that it is
	contained in a chain control set, hence in $E_{c}^{1}$. The intersection in
	(\ref{intersection}) is nontrivial if and only if $L(0)$ is nontrivial, i.e.,
	if $A$ is nonhyperbolic.
	
	We proceed to discuss several simple examples of linear control systems.
	Recall that they generate split affine control flows.
	
	\begin{example}
		\label{Example_hyp1}Consider the linear control system
		\begin{equation}
			\left(
			\begin{array}
				[c]{c}%
				\dot{x}\\
				\dot{y}%
			\end{array}
			\right)  =\left(
			\begin{array}
				[c]{cc}%
				1 & 0\\
				0 & -1
			\end{array}
			\right)  \left(
			\begin{array}
				[c]{c}%
				x\\
				y
			\end{array}
			\right)  +\left(
			\begin{array}
				[c]{c}%
				1\\
				1
			\end{array}
			\right)  u(t)\mbox{ with }u(t)\in\Omega=[-1,1]. \label{Ex6.2b}%
		\end{equation}
		The system is hyperbolic and the Lyapunov spaces of the linear part are
		$L(-1)=\{0\}\times\mathbb{R}$ and $L(1)=\mathbb{R}\times\{0\}$. The subbundles
		$\mathcal{V}_{1}=\mathcal{U}\times L(-1)$ and $\mathcal{V}_{2}=\mathcal{U}%
		\times L(1)$ yield the Selgrade bundles at infinity $\mathcal{V}_{1}^{\infty
		}=\mathcal{V}_{1}\times\{0\}$ and $\mathcal{V}_{2}^{\infty}=\mathcal{V}%
		_{2}\times\{0\}$ with associated chain recurrent components in the projective
		Poincar\'{e} bundle $\mathcal{U}\times\mathbb{P}^{2,0}\subset\mathcal{U}%
		\times\mathbb{P}^{2}$ given by%
		\[
		\mathcal{M}_{1}^{1}=\mathbb{P}\mathcal{V}_{1}^{\infty}=\mathcal{U}%
		\times\left\{  \mathbb{P}(0,1,0)\right\}  ,\mathcal{M}_{2}^{1}=\mathbb{P}%
		\mathcal{V}_{2}^{\infty}=\mathcal{U}\times\mathbb{P}(1,0,0),
		\]
		respectively. Inspection of the phase portrait in $\mathbb{R}^{2}$ shows that
		the unique chain control set is the compact set $E=\left[  -1,1\right]
		\times\lbrack-1,1]$ and for every $u\in\mathcal{U}$ the unique bounded
		solution $e(u,\cdot)$ is contained in $E$. As indicated in
		(\ref{chain_transitive1}) the lift of $E$ to $\mathcal{U}\times\mathbb{R}^{2}$
		is a maximal chain transitive set $\mathcal{E}$. Corollary
		\ref{Corollary_affine1}(iv) implies that the central Selgrade bundle has the
		form (\ref{6.7}) with $\dim\mathcal{V}_{c}^{1}=1$ and projection
		$\mathcal{M}_{c}^{1}=\mathbb{P}\mathcal{V}_{c}^{1}\subset\mathcal{U}%
		\times\mathbb{P}^{2,1}$. Furthermore, the chain control set satisfies
		$E_{c}^{1}=\mathbb{P}\left(  E\times\{1\}\right)  $.
	\end{example}
	
	The following linear control system is nonhyperbolic.
	
	\begin{example}
		\label{Example_nonhyperbolic}Consider
		\[
		\left(
		\begin{array}
			[c]{c}%
			\dot{x}\\
			\dot{y}%
		\end{array}
		\right)  =\left(
		\begin{array}
			[c]{cc}%
			0 & 0\\
			0 & 1
		\end{array}
		\right)  \left(
		\begin{array}
			[c]{c}%
			x\\
			y
		\end{array}
		\right)  +\left(
		\begin{array}
			[c]{c}%
			1\\
			1
		\end{array}
		\right)  u(t)\mbox{ with }u(t)\in\Omega=[-1,1].
		\]
		Here the Lyapunov spaces of the linear part are $L(0)=\mathbb{R}\times\{0\}$
		and $L(1)=\{0\}\times\mathbb{R}$. With $\mathcal{V}_{1}=\mathcal{U}\times
		L(0)$ and $\mathcal{V}_{2}=\mathcal{U}\times L(1)$ this yields the subbundles
		at infinity $\mathcal{V}_{1}^{\infty}=\mathcal{V}_{1}\times\{0\}$ and
		$\mathcal{V}_{2}^{\infty}=\mathcal{V}_{2}\times\{0\}$ with associated chain
		recurrent components in $\mathcal{U}\times\mathbb{P}^{2,0}$ given by
		\[
		\mathcal{M}_{1}^{1}=\mathbb{P}\mathcal{V}_{1}^{\infty}=\mathcal{U}%
		\times\left\{  \mathbb{P}(1,0,0)\right\}  ,\mathcal{M}_{2}^{1}=\mathbb{P}%
		\mathcal{V}_{2}^{\infty}=\mathcal{U}\times\mathbb{P}(0,1,0),
		\]
		respectively. By Corollary \ref{Corollary_split} and Remark
		\ref{Remark_linear} the central Selgrade bundle $\mathcal{V}_{c}^{1}$ has
		dimension $\dim\mathcal{V}_{c}^{1}=1+\dim L(0)=2$. Thus $\mathcal{V}%
		_{1}^{\infty}\subset\mathcal{V}_{c}^{1}$ and $\mathcal{V}_{2}^{\infty}$ is a
		Selgrade bundle at infinity. Inspection of the phase portrait in
		$\mathbb{R}^{2}$ shows that the unique chain control set $E$ is given by the
		strip $E=\mathbb{R}\times\lbrack-1,1]$. The lift of the chain control set $E$
		is the maximal chain transitive set $\mathcal{E}$ for the affine control flow
		$\Psi$. The orthogonal projection of the system on the northern hemisphere
		$\mathbb{S}^{2,+}$ to the unit disk yields the phase portrait for $u=\pm1$ and
		the chain control set $E_{c}^{1}=\mathbb{P}(E\times\{1\})$ sketched in
		\textbf{Fig.} \ref{figu1}. The Morse spectra satisfy%
		\[
		\Sigma_{Mo}(\mathcal{V}_{1}^{\infty};\Psi^{1})=\Sigma_{Mo}(L(0))=\{0\},\Sigma
		_{Mo}(\mathcal{V}_{2}^{\infty};\Psi^{1})=\Sigma_{Mo}(L(1))=\{1\},
		\]
		and the inclusion in Theorem \ref{Theorem_spectrum}(ii) implies $0\in
		\Sigma_{Mo}(\mathcal{V}_{c}^{1};\Psi^{1})$.
	\end{example}
	
	In the next example the eigenvalue $0$ of the matrix $A$ is not semisimple.
	
	\begin{example}
		\label{Example_6.11}Consider
		\[
		\left(
		\begin{array}
			[c]{c}%
			\dot{x}\\
			\dot{y}%
		\end{array}
		\right)  =\left(
		\begin{array}
			[c]{cc}%
			0 & 1\\
			0 & 0
		\end{array}
		\right)  \left(
		\begin{array}
			[c]{c}%
			x\\
			y
		\end{array}
		\right)  +\left(
		\begin{array}
			[c]{c}%
			1\\
			0
		\end{array}
		\right)  u(t)\mbox{ with }u(t)\in\Omega=[-1,1].
		\]
		The Lyapunov space is $L(0)=\mathbb{R}^{2}$ which is chain transitive for
		$u\equiv0$: The $x$-axis consists of equilibria and for $y\not =0$ all
		trajectories move on parallels to the $x$-axis (to the right for $y>0$ and to
		the left for $y<0$). Thus the chain control set $E$ coincides with
		$\mathbb{R}^{2}$. Note that the $(\varepsilon,T)$-chains become unbounded for
		$T\rightarrow\infty$. On the equator $\mathbb{S}^{1}\times\{0\}$ of
		$\mathbb{S}^{2}$ there are two equilibria given by the intersection with the
		eigenspace $\mathbb{R}\times\{0\}$. For $\Psi^{1}$ there is no Selgrade bundle
		at infinity and the central Selgrade bundle is $\mathcal{V}_{c}^{1}%
		=\mathcal{U}\times\mathbb{R}^{2}\times\mathbb{R}$. This yields the chain
		recurrent component $\mathcal{M}_{c}^{1}=\mathbb{P}\mathcal{V}_{c}%
		^{1}=\mathcal{U}\times\mathbb{P}^{2}$, and the chain control set on the
		projective Poincar\'{e} space is $E_{c}^{1}=\mathbb{P}^{2}$.
	\end{example}
	
	For a linear flow $\Phi$ Remark \ref{Remark_linear2} characterizes the central
	Selgrade bundle using the subbundles $\mathcal{V}_{i}$ such that
	$h^{1}(\mathcal{V}_{i})$ is chain transitive. The following example of a
	bilinear control system shows that there may exist several subbundles
	$\mathcal{V}_{i}$ with this property; cf. \cite[Example 5.2]{Col23} which is
	based on \cite[Example 5.5.12]{ColK00}.
	
	\begin{example}
		\label{Example6.8}Consider the bilinear control system given by
		\[
		\left(
		\begin{array}
			[c]{c}%
			\dot{x}\\
			\dot{y}%
		\end{array}
		\right)  =\left[  \left(
		\begin{array}
			[c]{cc}%
			0 & -\frac{1}{4}\\
			-\frac{1}{4} & 0
		\end{array}
		\right)  +u_{1}\left(
		\begin{array}
			[c]{cc}%
			1 & 0\\
			0 & 1
		\end{array}
		\right)  +u_{2}\left(
		\begin{array}
			[c]{cc}%
			0 & 1\\
			0 & 0
		\end{array}
		\right)  +u_{3}\left(
		\begin{array}
			[c]{cc}%
			0 & 0\\
			1 & 0
		\end{array}
		\right)  \right]  \left(
		\begin{array}
			[c]{c}%
			x\\
			y
		\end{array}
		\right)
		\]
		with
		\[
		u(t)=(u_{1}(t),u_{2}(t),u_{3}(t))\in\Omega=\left[  -1,1\right]  \times\left[
		-1/4,1/4\right]  \times\left[  -1/4,1/4\right]  \subset\mathbb{R}^{3}\mbox{.}
		\]
		This defines a linear flow $\Phi_{t}(u,x)=\left(  u(t+\cdot),\varphi
		(t,x,u)\right)  ,t\in\mathbb{R}$, on the vector bundle $\mathcal{U}%
		\times\mathbb{R}^{2}$. With
		\[
		A_{1}=\left\{  \left(
		\begin{array}
			[c]{c}%
			x\\
			\alpha x
		\end{array}
		\right)  \left\vert \alpha\in\left[  -\sqrt{2},-\frac{1}{\sqrt{2}}\right]
		\right.  \right\}  ,\,A_{2}=\left\{  \left(
		\begin{array}
			[c]{c}%
			x\\
			\alpha x
		\end{array}
		\right)  \left\vert \alpha\in\left[  \frac{1}{\sqrt{2}},\sqrt{2}\right]
		\right.  \right\}  ,
		\]
		the Selgrade bundles are, for $i=1,2$,
		\[
		\mathcal{V}_{i}=\left\{  (u,x,y)\in\mathcal{U}\times\mathbb{R}^{2}\left\vert
		\varphi(t,x,y,u)\in A_{i}\mbox{ for }t\in\mathbb{R}\right.  \right\}  \mbox{
			with }\dim\mathcal{V}_{i}=1.
		\]
		One obtains the two chain recurrent components $\mathcal{M}_{i}=\mathbb{P}%
		\mathcal{V}_{i}$ of $\mathbb{P}\Phi$ on the projective bundle $\mathcal{U}%
		\times\mathbb{P}^{1}$ ordered by $\mathcal{M}_{1}\preceq\mathcal{M}_{2}$. The
		Morse spectra are
		\[
		\mathbf{\Sigma}_{Mo}(\mathcal{V}_{1})=\left[  -2,1/2\right]  \mbox{ and
		}\mathbf{\Sigma}_{Mo}(\mathcal{V}_{2})=\left[  -1/2,2\right]  .
		\]
		Since, for $i=1,2$, one has $0\in\mathrm{int}\Sigma_{Mo}(\mathcal{V}_{i})$,
		Theorem \ref{Corollary4.4} implies that $h^{1}(\mathcal{V}_{i})$ is chain
		transitive in the projective Poincar\'{e} bundle $\mathcal{U}\times
		\mathbb{P}^{2}$. By Remark \ref{Remark_linear2} the central Selgrade bundle
		is
		\[
		\mathcal{V}_{c}^{1}=\mathcal{P}\oplus\mathcal{V}_{1}^{\infty}\oplus
		\mathcal{V}_{2}^{\infty}=\mathcal{U\times}\mathbb{R}^{3},
		\]
		where $\mathcal{P=U}\times\{0\}\times\mathbb{R}\subset\mathcal{U}%
		\times\mathbb{R}^{2}\times\mathbb{R}$ is the polar bundle. There is no
		Selgrade bundle at infinity.
	\end{example}
	
	\textbf{Acknowledgement.} We appreciate the careful reading and the
	constructive comments of an anonymous reviewer which helped to improve the paper.
	
	\section{Declarations}
	
	\subsection{Funding}
	
	No funding received
	
	\subsection{Conflict of interest/Competing interests}
	
	No interests of a financial or personal nature
	
	\subsection{Ethics approval}
	
	Not applicable
	
	\subsection{Consent to participate}
	
	According to \newline
	https://www.springer.com/gp/editorial-policies/informed-consent this item it
	is not applicable
	
	\subsection{Consent for publication}
	
	According to \newline
	https://www.springer.com/gp/editorial-policies/informed-consent this item it
	is not applicable
	
	\subsection{Availability of data and materials}
	
	Not applicable
	
	\subsection{Code availability}
	
	Not applicable
	
	\subsection{Authors' contributions}
	
	All authors contributed to all sections. All authors reviewed the final manuscript.

	\begin{figure}[h]
		\centering
		\includegraphics[width=10cm]{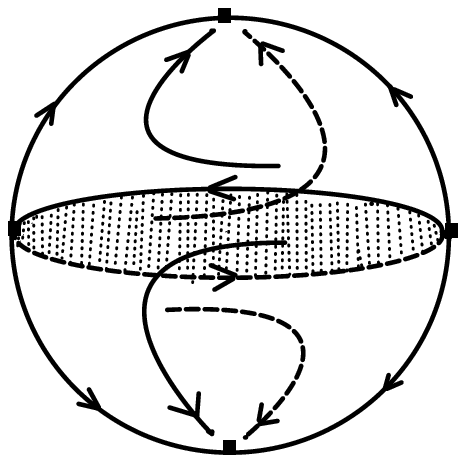}
		\caption{Chain control set $E_{c}^{1}$ and phase portraits for $u=\pm1$ in Example \ref{Example_nonhyperbolic}.} \label{figu1}
	\end{figure}

\end{document}


\title{Erratum to \textquotedblleft Chain recurrence and Selgrade's theorem for
affine flows\textquotedblright}
\date{}
\author{Fritz Colonius$^{1}$ $\,\ $ and $\,\ $ Alexandre J. Santana$^{2}$ \vspace{0cm}\\$^{1}$Institut f\"{u}r Mathematik, Universit\"{a}t Augsburg,
Universit\"{a}tsstrasse 2,\\Augsburg, 86159, Germany. fritz.colonius@uni-a.de\\$^{2}$Department of Mathematics, State University of Maringa, \\Av. Colombo, 5790, Maringa, 87020-900, Brazil. ajsantana@uem.br}
\maketitle

\begin{abstract}
For uniformly hyperbolic affine flows, a characterization of the chain
recurrent set is provided. This yields formulas for the central Selgrade
bundle of the lift to a linear flow.

\end{abstract}

Let $\Psi$ be an affine flow with uniformly hyperbolic linear part on a vector
bundle $B\times H$, where the induced flow on the compact metric space $B$ is
chain transitive and $H$ is a finite dimensional Hilbert space. For every
$b\in B$, there exists a unique bounded solution $(b\cdot t,e(b,t)),t\in
\mathbb{R}$. A conjugating map for the affine flow $\Psi$ and its linear part
$\Phi$ is given by $h_{aff}=(id_{B},h_{aff}^{0})$ with $h_{aff}^{0}%
(b,x)=x-e(b,0)$, and hence%
\[
h_{aff}(\Psi_{t}(b,x))=\Phi_{t}(h_{aff}(b,x))\text{ for }t\in\mathbb{R}%
,(b,x)\in B\times H.
\]
\cite[Theorem 25]{ColS25} asserts that $h_{aff}$ maps the chain recurrent set
$\mathcal{R}=B\times\{0\}$ of $\Phi$ onto the chain recurrent set of $\Psi$.
This is false. With the inverse $h_{aff}^{0}(b,\cdot)^{-1}%
(x)=x+e(b,0),\allowbreak x\in H$, of $h_{aff}^{0}(b,\cdot):H\rightarrow H$,
the chain recurrent set of $\Psi$ is
\[
h_{aff}^{-1}(B\times\{0\})=\left\{  (b,h_{aff}^{0}(b,\cdot)^{-1}(0)\left\vert
b\in B\right.  \right\}  =\left\{  (b,e(b,0))\left\vert b\in B\right.
\right\}  .
\]
For the lift of the affine flow $\Psi$ to a linear flow $\Phi^{1}$, one
obtains the following consequence. In \cite[Theorem 26(i)]{ColS25} the correct
formula for the central Selgrade bundle $\mathcal{V}_{c}^{1}$ of $\Phi^{1}$ is%
\[
\mathcal{V}_{c}^{1}=\left\{  (b,re(b,0),r)\in B\times H\times\mathbb{R}%
\left\vert b\in B,r\in\mathbb{R}\right.  \right\}  .
\]
This yields for the central Selgrade bundle of the lifted control flow in
\cite[Corollary 32(iv)]{ColS25}%
\[
\mathcal{V}_{c}^{1}=\left\{  (u,re(u,0),r)\in\mathcal{U}\times\mathbb{R}%
^{d}\times\mathbb{R}\left\vert u\in\mathcal{U},r\in\mathbb{R}\right.
\right\}  .
\]